\documentclass{amsbook}
\usepackage[usenames,dvipsnames]{color}

\usepackage{latexsym, amsmath, amscd, amssymb, amsthm, bm, ushort}
\usepackage[colorlinks=true,linkcolor=blue,urlcolor=blue, citecolor=blue]%
  {hyperref}
\usepackage[shortalphabetic]{amsrefs}
\usepackage[T1]{fontenc}
\usepackage{mathptmx}
\usepackage{microtype}
\usepackage{verbatim}
\usepackage[all]{xypic,xy}
\usepackage{graphicx}

  \usepackage{tikz,ifthen}
  \tikzstyle{perspective adjusted}=[%
x={(.8,-.07)},%
y={(0,1)},%
z={(-1.3,-.2)}]

\tikzstyle{perspective nqgs}=[%
x={(.8,.3)},%
y={(.2,1.5)},%
z={(2.4,-.5)}]

\newtheorem{lemma}{Lemma}[section]
\newtheorem{theorem}[lemma]{Theorem}
\newtheorem{corollary}[lemma]{Corollary}
\newtheorem{proposition}[lemma]{Proposition}
\newtheorem{conjecture}[lemma]{Conjecture}

\newtheorem{definition}[lemma]{Definition}

\theoremstyle{definition}

\newtheorem{remark}[lemma]{Remark}
\newtheorem{example}[lemma]{Example}
\theoremstyle{remark}
\newtheorem*{characterization1}{Characterization 1}
\newtheorem*{characterization2}{Characterization 2}
\newtheorem*{characterization3}{Characterization 3}

\newcommand{\CC}{\ensuremath{\mathbb{C}}} 
 
\newcommand{\NN}{\ensuremath{\mathbb{N}}} 
 
\newcommand{\QQ}{\ensuremath{\mathbb{Q}}} 
\newcommand{\RR}{\ensuremath{\mathbb{R}}} 
 
\newcommand{\ZZ}{\ensuremath{\mathbb{Z}}} 

\newcommand{\pr}{\ensuremath{\mathbb{P}}} 

\newcommand{\A}{\ensuremath{\mathcal{A}}}
\newcommand\transpose{{\rm T}}
\newcommand\el{\mathcal L}
\newcommand\osc[1]{{\Bbb Osc }^{#1}}
\newcommand{\adP}[1]{P^{(#1)}}

\newcommand\up[1]{\left\lceil #1 \right\rceil}

\newcommand{\pro}[2]{\langle #1, #2 \rangle}
\newcommand{\ad}[2]{{#1}^{(#2)}}
\newcommand{\nf}{\tau}
\newcommand{\NF}{\mathcal{N}}
\providecommand{\core}{\mathop{\rm core}\nolimits}

\newcommand{\into}{\hookrightarrow}
\renewcommand{\geq}{\geqslant}
\renewcommand{\leq}{\leqslant}

\newcommand{\iso}{\cong}

\DeclareMathOperator{\conv}{Conv}
\DeclareMathOperator{\cod}{codim}

\DeclareMathOperator{\Vol}{Vol}
\DeclareMathOperator{\Eu}{Eu}

\DeclareMathOperator{\cd}{codeg}
\DeclareMathOperator{\Cayley}{Cayley}
\DeclareMathOperator{\Span}{Span}
\newcommand{\Sig}{\Sigma}

\begin{document}

\title[Linear toric fibrations]{Linear Toric Fibrations}

\author[S.~Di~Rocco]{Sandra Di Rocco}
\address{Sandra Di Rocco\\ Department of Mathematics\\ Royal Institute of
  Technology (KTH)\\ 10044 Stockholm\\ Sweden}
\email{\href{mailto:dirocco@kth.se}{dirocco@kth.se}}
\urladdr{\href{http://www.math.kth.se/~dirocco}%
  {www.math.kth.se/~dirocco}}


\date{Fall 2013}

\begin{abstract}
Polarized toric varieties which are birationally equivalent  to projective toric bundles are associated to a class of polytopes called Cayley polytopes. Their geometry and combinatorics have a fruitful interplay leading to fundamental insight in both directions. These
notes will illustrate geometrical phenomena, in algebraic geometry and
neighboring fields, which are characterized by a Cayley structure. Examples
are projective duality of toric varieties and polyhedral adjunction theory.

\vspace{.3in}

\noindent
{\bf Acknowledgments.} The author was supported by a  grant from the
Swedish Research Council (VR). Special thanks to A. Lundman, B. Nill and B. Sturmfels for 
reading a preliminary version of the notes.\\
\ \\
\hfill Stockholm, Fall 2013\\ 
\hfill Sandra Di Rocco

\end{abstract}

\maketitle
\tableofcontents

\section{Introduction}
These notes are based on three lectures given at the 2013 CIME/CIRM summer school
{\em Combinatorial Algebraic Geometry}. 

 The purpose of this series of lectures is to introduce the notion of a {\em toric fibration} and to give  its geometrical and combinatorial characterizations. 
 
Toric fibrations $f:X\to Y,$ together with a choice of an ample line bundle $L$ on $X$ are associated to  convex polytopes  called  {\em Cayley sums.} Such a polytope is a convex polytope $P\subset \RR^n$  obtained by assembling a number of lower dimensional polytopes $R_i,$ whose normal fan defines the same toric variety $Y.$ Let $\RR^n=M\otimes \RR,$ for a lattice $M.$ The building-blocks $R_i$ are glued together following their image via   a surjective map of lattices $\pi:M\to\Lambda,$ see Definition \ref{defCayley}.  In particular the normal fan of the polytope $\pi(P)$ defines the generic fibre of the map $f.$  We will denote Cayley sums  by ${\rm Cayley}(R_0,\ldots,R_t)_{\pi,Y}.$
Our aim is to illustrate how classical notions in projective geometry are captured by certain properties of the associated Cayley sum. 

When the image polytope $\pi(P)$ is a unimodular simplex $\Delta_k$ the generic fibre of the fibration $f$ is a projective space $\pr^k$ embedded linearly, i.e. $L|_F={\mathcal O}_{\pr^k}(1).$ For this reason the fibration is called  a {\it linear toric fibration}.  The following  example  illustrates a linear toric fibration and the representation of the associated polytope as a Cayley sum.\begin{center}
{\small
\begin{tikzpicture}

\fill (0,0,0) circle (2pt) node[above]{} ;
\fill (1,0,0) circle (1pt) node[above]{} ;
\fill (2,0,0) circle (1pt) node[above]{} ;
\fill (3,0,0) circle (2pt) node[above]{} ;
\fill (0,0,1) circle (1pt) node[above]{} ;
\fill (0,0,2) circle (1pt) node[above]{} ;
\fill (0,0,3) circle (2pt) node[above]{} ;
\fill (1,0,1) circle (1pt) node[above]{} ;
\fill (1,0,2) circle (1pt) node[above]{} ;
\fill (1,0,3) circle (1pt) node[above]{} ;
\fill (2,0,1) circle (1pt) node[above]{} ;
\fill (2,0,2) circle (1pt) node[above]{} ;
\fill (2,0,3) circle (1pt) node[above]{} ;
\fill (3,0,1) circle (1pt) node[above]{} ;
\fill (3,0,2) circle (1pt) node[above]{} ;
\fill (3,0,3) circle (2pt) node[above]{} ;
\fill (0,1,0) circle (2pt) node[above]{} ;
\fill (1,1,0) circle (2pt) node[above]{} ;
\fill (0,1,1) circle (2pt) node[above]{} ;
\fill (1,1,1) circle (2pt) node[above]{} ;
\draw[->]  (3.5,0,0) -- (4, 0,0);
\node [above] at (3.8,0.1,0) {$\pi$};
\draw (0,1,0)--(1,1,0);
\draw (1,1,0)--(1,1,1);
\draw (1,1,1)--(0,1,1);
\draw (0,1,0)--(0,1,1);
\draw[dashed] (0,1,0)--(0,0,0);
\draw (0,1,1)--(0,0,3);
\draw (1,1,1)--(3,0,3);
\draw (1,1,0)--(3,0,0);
\draw (0,0,3)--(3,0,3);
\draw (3,0,3)--(3,0,0);
\draw[dashed] (3,0,0)--(0,0,0);
\draw[dashed] (0,0,0)--(0,0,3);
\fill (5,2,0) circle (1pt) node[above]{} ;
\fill (5,0,0) circle (2pt) node[above]{} ;
\fill (5,1,0) circle (2pt) node[above]{} ;
\fill (5,-1,0) circle (1pt) node[above]{} ;
\draw (5.,0,0)--(5,1,0);
\node [above] at (-4,0,0) {$\Small f:\pr({\mathcal O}_{\pr^1\times\pr^1}(1,1)\oplus {\mathcal O}_{\pr^1\times\pr^1}(3,3))\to \pr^2$} ;

\end{tikzpicture}}
\end{center}

Section \ref{Lec1} will be devoted to define these concepts and to give the most relevant examples. In the following two sections we will present two characterizations of Cayley sums corresponding to linear toric fibrations.  In both cases there are rich and interesting connections with classical projective geometry.
 
 Section \ref{Lec2} discusses discriminants of polynomials. A polynomial supported on a subset $\A\subset\ZZ^n$ is a polynomial in $n$ variables $x=(x_1,\ldots,x_n)$ of the form $p_\A=\sum_{a\in\A} c_ax^{a}.$ The $\A$-discriminant is again a polynomial in $|\A|$ variables, $\Delta_\A(c_a),$ vanishing whenever the corresponding polynomial had at least one singularity in the torus $(\CC^*)^n.$ Understanding the existence and in that case the degree of the discriminant polynomial, for given classes of point-configurations $\A$, is highly desirable.
Finite subsets $\A\subset\ZZ^n$ define toric projective varieties, $X_\A\subset \pr^{|\A|-1}.$ It is classical in Algebraic Geometry to associate to a given embedding, $X\subset \pr^m,$ the variety parametrizing hyperplanes singular along $X.$ This variety is called the dual variety and it is denoted by $X^\vee.$ Understanding when the codimension of the dual variety is higher that one and giving efficient formulas for its degree is a long standing problem. We will see that projective duality is a useful  tool for describing  the discriminants $\Delta_\A$ when the associated polytope $\conv(\A)$ is smooth or simple. In fact the case when $\Delta_\A=1$ is completely characterized by Cayley sums and thus by toric fibrations. 

In the non singular case the following holds.
 \begin{characterization1} If $P_\A=\conv(\A)$ is a smooth polytope then 
 the following assertions  are equivalent:
 \begin{enumerate}
  \item $P_\A={\rm Cayley}_{\pi,Y}(R_0,\ldots,R_t)$ with $t\geq max(2, \frac{n+1}{2}).$
  \item $\cod(X_\A^\vee)>1.$
  \item $\Delta_\A=1.$
\end{enumerate}
\end{characterization1} 

When the codimension of $X_\A^\vee$ is one then its degree is given by an alternating sum of volumes of the faces of the polytope $P_\A.$ We will see that this formula corresponds the the top Chern class of the so called first jet bundle. This interpretation has a useful consequence. When the codimension of $X_\A^\vee$ is higher than one this Chern class has to vanish. This leads to another characterization of Cayley sums.

 \begin{characterization2} If $P_\A=\conv(\A)$ is a smooth polytope then 
 the following assertions  are equivalent:
 \begin{enumerate}
  \item $P_\A={\rm Cayley}_{\pi,Y}(R_0,\ldots,R_t)$ with $t\geq max(2, \frac{n+1}{2}).$
  \item $\sum_{\emptyset\neq F\prec P_\A} (-1)^{\cod(F)} (\dim(F)+1)!\Vol(F)=0.$
\end{enumerate}
\end{characterization2}

 In Section \ref{Lec3} we discuss  the problem of  classifying convex polytopes and algebraic varieties. A classification is typically done via 
 invariants.
  In recent years much attention has been concentrated on  the notion of codegree of a convex polytope. 
  $$\cd(P)=min_\ZZ\{ t | tP\text{ has interior lattice points}\}$$
The unimodular simplex for example has $\cd(\Delta_n)=n+1.$ Batyrev and Nill conjectured that imposing this invariant to be large should force the polytope to be a Cayley sum.

It turned out that a $\QQ$-version of this invariant, what we denote by $\mu(P)$, corresponds to a classical invariant in classification theory of algebraic varieties, called the log-canonical threshold. Let $(X_P,\el_P)$ be the toric variety  and ample line bundle associated to the polytope $P.$ The canonical threshold $\mu(\el_P)$ and the nef-value $\tau(\el_P)$ are the invariants used heavily in the classification theory of Gorenstein algebraic varieties. In particular Beltrametti-Sommese-Wisniewski conjectured that imposing $\mu(\el_P)$ to be large should force the variety to have the structure of a fibration.

Again in the toric setting we will see that these two stories intersect making it possible to
prove the above conjectures, at least in the smooth case, and leading to yet another characterization of Cayley sums.

 \begin{characterization3}  Let $P$ be a smooth polytope.
 The following assertions  are equivalent:
 \begin{enumerate}
\item $\cd(P)\geq (n+3)/2.$
\item $P$ is  isomorphic to a Cayley sum ${\rm Cayley}(R_0,\ldots,R_t)_{\pi,Y}$ where $t+1=\cd(P)$ with $k>\frac{n}{2}.$
\item $\mu(\el_P)=\tau(\el_P)\geq (n+3)/2.$
\end{enumerate}
  \end{characterization3}

In fact the characterizations  above extend to more general classes of polytopes, not necessarily smooth, as we explain in Sections \ref{Lec2} and \ref{Lec3}.
Section \ref{final} is devoted to give a complete proof of these characterizations.
\section{Conventions and notation}
We assume basic knowledge of toric geometry and refer to \cite{EW,FU,ODA} for the necessary background on toric varieties. We will moreover assume some knowledge of projective algebraic geometry. We refer the reader to 
\cite{HA, FUb} for further details.
Throughout this paper, we work over the  field of complex numbers $\CC$. By a polarized variety we mean a pair $(X,L)$ where $X$ is an algebraic variety and $L$ is an ample line bundle on $X.$

\subsection{Toric geometry.} In this note a toric variety, $X,$ is always assumed to be normal and thus defined by a fan $\Sig_X\subset N\otimes\RR$ for a lattice $ N.$ By $\Sig_X(n)$ we will denote the collection 
of $n$-dimensional cones of $\Sig_X.$ The invariant sub-variety of codimension $t$ associated to a cone $\sigma\in\Sigma(t)$ will be denoted by $V(\sigma).$

For a lattice $\Delta$ we set $\Delta_\RR=\Delta\otimes_\ZZ \RR.$ We denote by $\Delta^{\vee}=Hom(\Delta,\ZZ)$ the dual lattice. If $\pi: \Delta\to \Gamma$ is a morphism of lattices we denote by $\pi_\RR:\Delta_\RR\to \Gamma_\RR$ the induced $\RR$-homomorphism.
By a lattice polytope  $P\subset \Delta_\RR$ we mean a polytope with vertices in $\Delta.$

Let $P\subset\RR^n$ be a lattice polytope of dimension $n.$ Consider the graded semigroup $\Pi_P$ generated by $(\{1\}\times P)\cap (\NN\times \ZZ^n).$ The polarized variety $(Proj(\CC[\Pi_P]),{\mathcal O}(1))$ is a toric variety associated to the polytope $P$. It will be sometimes denoted by $(X_P,L_P).$ Notice that the toric variety $X_P$ is defined by the (inner) normal fan of $P.$ Vice versa the symbol $P_{(X,L)}$ will denote the lattice polytope associated to a polarized toric variety $(X,L).$

Two polytopes are said to be {\it normally equivalent} if their normal fans are isomorphic. 

The symbol $\Delta_n$ denotes the smooth (unimodular) simplex of dimension $n.$ Recall that an $n$-dimensional polytope is {\em simple} if  through every vertex pass exactly $n$ edges. A lattice polytope is {\em smooth} if it is simple and the primitive vectors of the edges through every vertex form a lattice basis. Smooth polytopes are associates to smooth projective varieties. Simple polytopes are associated to $\QQ$-factorial projective varieties.

When the toric variety is defined via a point configuration $\A\subset\ZZ^n$ we will use the symbol $(X_\A,\el_\A)$ for the associated polarized toric variety and $P_\A=\conv(\A)$ for the associated polytope. The corresponding fan is denoted by $\Sigma_\A.$

\subsection{Vector bundles.} The notion of Chern classes of a vector bundle is an essential tool in some of the proofs. Let $E$ be a vector bundle of rank $k$ over an $n$-dimensional algebraic variety $X$. Recall that the $i$-th Chern class of $E,$  $c_i(E),$ is the class of a codimension $i$ cycle on $X$ modulo rational equivalence. The top Chern class of a rank $k\geq n$ vector bundle is $c_n(E).$  The same symbol $c_n(E)$ will be used to denote the degree of the associated zero-dimensional subvariety. 

The projectivization of a vector bundle plays a fundamental role throughout these notes. Let $S^l(E)$ denote the $l$-th symmetric power of a rank $r+1$ vector bundle $E.$ The projectivization of $E$ is 
$\pr(E)={\rm Proj}(\oplus_{l=0}^\infty S^l(E)).$ It is a projective bundle  with fiber $F=\pr(E)_x=\pr(E_x)=\pr^{r}.$ Let $\pi: \pr(E)\to Y$ be the bundle map. There is a line bundle $\xi$ on $\pr(E),$ called the tautological line bundle, defined by the property that   $\xi_F\cong {\mathcal O}_{\pr^r}(1).$ When $E$ is a vector bundle on a toric variety $Y$ then the projective bundle $\pr(E)$ has the structure of a toric variety if and only if $E=L_1\oplus\ldots\oplus L_k,$ \cite[Lemma 1.1.]{DRS04}. When the line bundles $L_i$ are ample then the tautological line bundle $\xi$ is also ample.

We refer to \cite{FU} for the necessary background on vector bundles and their characteristic classes.

\section{Toric fibrations}\label{Lec1}
\begin{definition}
A {\em toric fibration} is a surjective flat map $f:X\to Y$ with connected fibres  where
\begin{enumerate}
\item $X$ is a toric variety
\item $Y$ is a normal algebraic variety
\item $\dim(Y)<\dim(X).$
\end{enumerate}
\end{definition}
\begin{remark}
A surjective morphism $f : X \to Y$ , with connected
fibers between normal projective varieties, induces a homomorphism from the connected
component of the identity of the automorphism group of $X$ to the connected
component of the identity of the automorphism group of $Y$ , with respect to which
$f$ is equivariant.
It follows that if $f:X\to Y$ is a toric fibration then $Y$ and a general fiber $F$ 
admit the structure of  toric varieties such that $f$ becomes an equivariant morphism.
 Moreover if $X$ is smooth, respectively $\QQ$-factorial,  then   $Y $ and $F$ are also  smooth, respectively $\QQ$-factorial.
\end{remark}
\begin{example}\label{projectivebundle}
Let $L_0,\ldots, L_k$ be line bundles over a toric variety $Y$ such that the tautological line bundle $\xi_E$ defined by the equivariant vector bundle $E=L_0\oplus\ldots\oplus L_k$ is ample. The total space $\pr{}(L_0\oplus\ldots\oplus L_k)$ is a toric variety, see \cite[Lemma 1.1]{DRS04}, and the projective bundle $\pi: \pr{}(L_0\oplus\ldots\oplus L_k)\to Y$ is a toric fibration. 
\end{example}

 \subsection*{Combinatorial characterization} A toric fibration has  the following combinatorial characterization, see \cite[Chapter VI]{EW} for further details.
 Let $N\iso\ZZ^n,$ be a lattice $\Sig\subset N\otimes\RR$ be a fan and $X=X_\Sig,$ the associated toric variety. Let $i:\Delta\into N$ be  a sub-lattice.

\begin{proposition}\cite{EW}
The inclusion $i$ induces a toric fibration, $f:X\to Y$ if and only if:
\begin{enumerate}
\item $\Delta$ is a primitive lattice, i.e. $(\Delta\otimes\RR)\cap N=\Delta.$
\item For every $\sigma\in\Sig(n),$ $\sigma=\tau+\eta,$ where $\tau\in\Delta$ and $\eta\cap\Delta=\{0\}$
(i.e. $\Sig$ is a split fan).
\end{enumerate}
\end{proposition}

We briefly outline the construction. The projection $\pi: N\to N/\Delta$ induces  a map of fans $\Sig \to \pi(\Sig)$ and thus a map of toric varieties $f:X\to Y.$ The general fiber $F$ is a toric variety defined by the fan
$\Sig_F=\{\sigma\in\Sig\cap\Delta\}.$ The invariant varieties $V(\tau)$ in $X,$ where $\tau\in\Sig$ is a maximal-dimensional cone in $\Sig_F,$ are called invariant sections of the fibration.
The subvariety $V(\tau)$ is the invariant section passing through the fixpoint of $F$ corresponding to the cone $\tau\in\Sig_F.$  Observe that they are all isomorphic, as toric varieties, to $Y.$
\begin{example}
In  Example \ref{projectivebundle} let $\Gamma\subset \RR^n$ be the fan defining $Y,$ and let $D_1,\ldots,D_s$ be the generators of $Pic(Y)$ associated to the rays $\eta_i,\ldots,\eta_s\subset \Gamma.$  The line bundle $L_i$ can be written as $L_i=\sum \phi_i(\eta_j)D_j$ where $\phi_i: \Gamma_\RR\to \RR$ are piecewise linear functions.  Let $e_1,\ldots,e_k\in \ZZ^k$ be a lattice basis and let $e_0=-e_1-\ldots-e_k.$ One can define a map:
$$\psi:\RR^{n}\to \RR^{n+k} \text{ as } \psi(v)=(v, \sum \phi_i(v)e_i).$$
Consider now the  fan $\Sigma'\subset\RR^{n+k}$ given by the image of $\Gamma$ under $\psi:$
$\Sigma'=\{\psi(\sigma), \sigma\subset\Gamma\}.$ Let $\Pi\subset\ZZ^k$ be the fan defining $\pr^{k}.$
The fan $\Sigma=\{\sigma'+\tau | \sigma'\in\Sigma', \tau\in \Pi\}$
is a split fan, defining the toric fibration $\pi: \pr{}(L_0\oplus\ldots\oplus L_k)\to Y.$
 See also \cite[Proposition 1.33]{ODAb}.\end{example}
\begin{definition}
A {\em polarized toric fibration} is a triple $(f:X\to Y, L),$ where $f$ is a toric fibration and $L$ is an ample line bundle on $X.$
\end{definition}

 Observe that for a general fiber $F,$ the pair $(F, L|_F)$ is also a polarized toric variety. It follow that both pairs $(X,L)$ and $(F, L|_F)$  define lattice polytopes $P_{(X,L)}, P_{(F, L|_F)}$. The polytope $P_{(X,L)}$ is in fact a ``twisted sum" of a finite number of lattice polytopes fibering over $P_{(F, L|_F)}.$

\begin{definition}\label{defCayley}
Let $R_0,\ldots, R_k \subset M_\RR$ be  lattice polytopes and let $k\geq 1.$ Let $\pi:M\to\Lambda$ be a surjective map of lattices such that $\pi_\RR(R_i)=v_i$ and  such that$v_0,\cdots, v_k$ are distinct vertices of $\conv(v_0,\ldots,v_k).$ We will call a {\em Cayley $\pi$-twisted sum (or simply a Cayley sum) of $R_0,\ldots,R_k$} a polytope which is affinely isomorphic to $\conv(R_0,\ldots,R_k).$ 
$$\text{We will denote it by: }[R_0\star\ldots\star R_k]_\pi.$$

If the polytopes $R_i$ are additionally  normally equivalent, i.e. they define the same normal fan $\Sig_Y,$ we will
denote the Cayley sum by:
$$\Cayley(R_0,\ldots,R_k)_{(\pi,Y)}.$$
\end{definition}
 We will see that these are the polytopes that are associated to  polarized toric fibrations. 
 \begin{proposition}\cite{CDR08}
Let $X=X_\Sig$  be a toric variety of dimension $n,$ where $\Sig\subset N_\RR\iso\RR^n,$  and let $i:\Delta\into N$ be  a sublattice. Let $L$ be an ample line bundle on $X.$ Then the inclusion $i$ induces a  polarized toric fibration   $(f:X\to Y, L)$ if and only if $P_{(X,L)}=\Cayley(R_0,\ldots,R_k)_{(\pi,Y)},$ where $R_0,\ldots,R_k$ are normally equivalent polytopes on $Y$ and  $\pi:M\to\Lambda$ is the lattice map dual to $i.$
 \end{proposition}
 \begin{proof}
\noindent We first prove the implication "$\Rightarrow$".

Assume that $i\colon\Delta\hookrightarrow N$ 
induces a toric fibration $f\colon X\to Y$ and consider the polarization $L$ on $X.$ We will prove that $P_{(X,L)}=\Cayley(R_0,\ldots,R_k)_{(\pi,Y)}$ for some normally equivalent polytopes $R_0,\ldots,R_k.$ 

Notice first that the fact that $\Delta$ is a primitive sub-lattice of $N$ implies that the dual map $\pi:M\to\Lambda$ is a surjection.
Let $F$ be a general fiber of $f,$ and let 
$S:=P_{(F,L_{|F})}\subset\Lambda_{\RR}$. 
Denote by $v_0,\dotsc,v_k$ the vertices of
$S$. Every $v_i$ corresponds to a fixed point of $F$; call
$Y_i=V(\tau_i)$ the invariant section of $f$ passing through that
point. Note that $\tau_i\in\Sigma_X$, $\dim\tau_i=\dim F$ 
and $\tau_i\subset\Delta_{\RR}.$  Let $R_i$ be the face of $P_{(X,L)}$
corresponding to $Y_i$. 

Observe that $Aff(\tau_i)=\Delta_{\RR}$, so that
$Aff(\tau_i^{\perp})=\Delta_{\RR}^{\perp}=\ker\pi_{\RR}$.
 Then there exists $u_i\in M$ such that:
\begin{itemize}
\item $Aff(R_i)+u_i=\ker(\pi_{\RR})$;
\item $R_i+u_i=P_{(Y_i,L_{|Y_i})}$.
\end{itemize}
This says that $R_0,\dotsc,R_k$ 
are normally equivalent (because
every $Y_i$ is isomorphic to $Y$), and that $\pi_{\RR}(R_i)$ is a
point.  
Since the $Y_i$'s are pairwise disjoint, the same
holds for the $R_i$'s. If $s$ is the number of fixed points of
$Y$, then each $R_i$ has $s$ vertices. On the other hand, we know that
$F$ has $(k+1)$ fixed points, and therefore 
$X$ must
have $s(k+1)$ fixed points. 
So $P_{(X,L)}$ has $s(k+1)$ vertices, namely the
union of all vertices of the $R_i$'s. We can conclude that
$$P_{(X,L)}=\conv(R_0,\ldots,R_k)$$
Let  $D=\sum_{x\in \Sig(1)} a_x D_x$ be an invariant 
Cartier divisor on $X$ such that
$L={\mathcal O}_X(D)$. Since $F$ is a general fiber, 
we have
$D_x\cap F\neq\emptyset$ if and only if
$x\in\Delta$,  and $D_{|F}=\sum_{x\in \Delta} a_{x}
D_{x|F}$.  This implies  that
 $\pi_{\RR}(R_i)=v_i$ and $\pi_{\RR}(P_{(X,L)})=S.$  We conclude that
 $P_{(X,L)}=\Cayley(R_0,\ldots,R_k)_{(\pi,Y)}.$
\smallskip

We now show the other direction: "$\Leftarrow$". 

Assume that $P_{(X,L)}=\Cayley(R_0,\ldots,R_k)_{(\pi,Y)}.$ We will prove that the associated polarized toric variety is a polarized toric fibration. 
First observe that the fact that the dual map $\pi$ is a surjection implies that the sublattice $\Delta$ is primitive. Since $v_i$ is a vertex of $\pi_{\RR}(P_{(X,L)})$,
$R_i$ is a face of $P_{(X,L)}$ for every $i=0,\dotsc,k$.  Let $Y$ be the
projective toric variety defined by the polytopes  
$R_i$. Observe that $Aff(R_i)$ is a translate of $\ker\pi_{\RR}$, and
$(\ker\pi)^{\vee}=N/\Delta$. So the fan $\Sigma_Y$
is contained in $(N/\Delta)_{\RR}$.

Let  $\gamma\in\Sigma_Y(\dim(Y))$ and 
for every $i=0,\dotsc,k$ let $w_i$ be the vertex of $R_i$
corresponding to $\gamma$.
We will  show that
$Q:=\conv(w_0,\ldots,w_k)$ is a face of $P_{(X,L)}$. 

Observe first  that $(\pi_{\RR})_{Aff(Q)}\colon Aff(Q)\to
\Lambda_{\RR}$ is bijective. Let $H$ be the linear subspace
of $M_{\RR}$ which is a translate of $Aff(Q)$. Then we have
$M_{\RR}=H\oplus \ker\pi_{\RR}$.  Dually $N_{\RR}=\Delta_{\RR}\oplus
H^{\perp}$, where $H^{\perp}$ projects isomorphically onto
$(N/\Delta)_{\RR}$. 

Let $u\in H^{\perp}$ be such that its image in $(N/\Delta)_{\QQ}$ is
contained in the interior of $\gamma$.
 Then for every $i=0,\dotsc,k$
we have that (see \cite[\S 1.5]{FU}):
\[\begin{array}{cl}
 (u,x)\geq (u,w_i)&\text{ for every }x\in R_i, \\
 (u,x)= (u,w_i)&\text{ if and only if }x=w_i.$$
 \end{array}\]
 Moreover  $u$ is constant on $Aff(Q)$, namely
there exists $m_0\in\QQ$ such
that $(u,z)=m_0$ for every $z\in Q$.

Any $z\in P$ can be written as
$z=\sum_{i=1}^l \lambda_i
 z_i$, with $z_i\in R_i$, $\lambda_i\geq 0$ and
 $\sum_{i=0}^l\lambda_i=1$.  Then
$$({u},z)=\sum_{i=1}^l \lambda_i (u,z_i)
\geq\sum_{i=1}^l \lambda_i({u},w_i)=
\sum_{i=1}^l \lambda_i m_0=m_0.$$
Moreover, $({u},z)=m_0$ 
if and only if $\lambda_i> 0$ for every $i$ such that
$(u,z_i)=(u,w_i),$ and $\lambda_i=0$ otherwise. This happens if and only if $z\in Q,$ 
implying that $Q$ is a face of $P_{(X,L)}$.

Let $\sigma\in\Sigma_X$ be a cone of maximal dimension, and
let $w$ be the corresponding vertex of
$P_{(X,L)}$. Then $\pi(w)$ is a vertex, say $v_1$, 
of $\pi_{\QQ}(P_{(X,L)})$ and
hence $w$ lies in  
$R_1$. Since $R_{1}$ is also a face of $P_{(X,L)}$, $w$ is a vertex of
$R_1$ and hence it corresponds to a maximal dimensional cone in $\Sigma_Y.$  In each $R_i$, consider the vertex $w_i$ corresponding to
the same cone of $\Sigma_Y.$ 
We set $w_1=w$. We have shown that
$Q:=\conv(w_0,\ldots,w_k)$ 
is a face of $P_{(X,L)}$, and $w=Q\cap R_1$. Now call $\tau$ and $\eta$ the cones of $\Sigma_X$ corresponding
respectively to  $R_1$ and $Q$. It follows that $\sigma=\tau+\eta$,
$\tau\subset\Delta_{\QQ}$, and $\eta\cap\Delta_{\QQ}=\{0\}$.
This concludes the proof.
\end{proof}
The previous proof  shows the following corollary.
\begin{corollary}
Let $(f:X\to Y, L)$ be a  polarized toric fibration  and let  $P_{(X,L)}=\Cayley(R_0,\ldots,R_k)_{(\pi,Y)}$ be the associated polytope.  Let $F$ be a general fiber of the fibration, $Y_0,\ldots, Y_k$ be the invariant sections and $\pi(R_i)=v_i.$ The following holds.
\begin{enumerate}
\item The polarized toric variety $(F, L|_F)$ corresponds to the polytope $P_{(F, L|_F)}=\conv(v_0,\ldots, v_k).$
\item The polarized toric varieties $(Y_i, L_{Y_i})$ correspond to the polytopes 
$$R_0-u_0, \cdots, R_k-u_k,$$ where $u_i\in M$ is a point such that $\pi(u_i)=\pi(R_i).$
\end{enumerate}
\end{corollary}
\begin{example}
The toric surface obtained by blowing up  $\pr^2$ at a fixpoint has the structure of a toric fibration, $\pr({\mathcal O}_{\pr^1}\oplus {\mathcal O}_{\pr^1}(1))\to \pr^1.$ It is often referred to as the Hirzebruch surface ${\mathbb F}_1.$ Consider the polarization given  by the tautological line bundle $\xi=2\phi^*({\mathcal O}_{\pr^2}(1))-E$ where $\phi$ is the blow-up map and $E$ is the exceptional divisor. The associated polytope is $P=\Cayley(\Delta_1, 2\Delta_1),$ see the figure below. 

\begin{center}
\begin{tikzpicture}

\fill (-1,0) circle (1pt) node[above]{} ;
\fill (0,0) circle (1pt) node[above]{} ;
\fill (1,0) circle (1pt) node[above]{} ;
\fill (2,0) circle (1pt) node[above]{} ;
\fill (-1,1) circle (1pt) node[above]{} ;
\fill (0,1) circle (2pt) node[above]{} ;
\fill (1,1) circle (2pt) node[above]{} ;
\fill (2,1) circle (1pt) node[above]{} ;
\fill (-1,2) circle (1pt) node[above]{} ;
\fill (0,2) circle (1pt) node[above]{} ;
\fill (1,2) circle (2pt) node[above]{} ;
\fill (2,2) circle (1pt) node[above]{} ;
\fill (-1,3) circle (1pt) node[above]{} ;
\fill (0,3) circle (2pt) node[above]{} ;
\fill (1,3) circle (1pt) node[above]{} ;
\fill (2,3) circle (1pt) node[above]{} ;
\fill (-1,-2) circle (1pt) node[above]{} ;
\fill (0,-2) circle (2pt) node[above]{} ;
\fill (1,-2) circle (2pt) node[above]{} ;
\fill (2,-2) circle (1pt) node[above]{} ;
\draw[->]  (0.5,0) -- (0.5,-1);

\draw (0,1)--(1,1);
\draw (1,1)--(1,2);
\draw (1,2)--(0,3);
\draw (0,3)--(0,1);
\draw (0,-2)--(1,-2);

\node [above] at (0.7,-0.5) {$\pi$};
\node [above] at (1.3,1.5) {$\Delta_1$};
\node [above] at (-0.3,1.5) {$2\Delta_1$};
\node [above] at (0, -1.9) {$v_1$};
\node [above] at (1,-1.9) {$v_0$};
\end{tikzpicture}
\end{center}
\end{example}

\begin{remark}
The following are important classes of polarized toric fibrations, relevant both  in Combinatorics and Algebraic Geometry. 

\noindent{\bf Projective bundles.} When $\pi(P)=\Delta_t$ the polytope $\Cayley(R_0,\ldots,R_t)_{(\pi,Y)}$ defines the polarized toric fibration $(\pr(L_0\oplus\ldots\oplus L_t)\to Y, \xi),$ where the $L_i$ are ample line bundles on the toric variety $Y$  and $\xi$ is the tautological line bundle.  In particular $L|_F={\mathcal O}_{\pr^t}(1).$ These  fibrations play an important role in the theory of discriminants and resultants of polynomial systems. See  Section \ref{Lec2}  for more details.

\noindent{\bf Mori fibrations.} When $\pi(P)$ is a simplex (not necessarily smooth) the Cayley polytope  $\Cayley(R_0,\ldots,R_k)_{(\pi,Y)}$ defines a Mori fibration, i.e.  a surjective flat map onto a $\QQ$-factorial toric variety whose generic fibre is reduced and has Picard number one. This type of fibrations  are important blocks in the Minimal Model Program for toric varieties. See 
\cite{CDR08} and \cite{Re83} for more details.

\noindent{\bf $\pr^k$-bundles.} When $\pi(P)=k\Delta_t$ then again the variety has the structure of a $\pr^t$-fibration whose general fiber $F$ is embedded as an $k$-Veronese embedding: $(F,L|_F)=(\pr^t, {\mathcal O}_{\pr^t}(k)).$ These fibrations arise in the study of $k$-th toric duality, see \cite{DDRP12}.
\end{remark}

In the polarized toric fibration $(\pr(L_0\oplus\ldots\oplus L_t),\xi)$ the fibres are embedded as linear spaces. For this reason the associated Cayley polytopes $\Cayley(R_0,\ldots,R_t)_{(\pi,Y)}$ can be referred to as {\it linear toric fibrations}.

 \begin{remark} For general Cayley sums,  $[R_0\star\ldots\star R_k]_\pi,$ one has the following geometrical interpretation. Let $(X,L)$ be the associated polarized toric variety and let $Y$ be the toric variety defined by the Minkowski sum $R_0+\ldots+R_k.$ The fan defining $Y$ is a refinement of the normal fans of the $R_i$ for $i=0,\ldots,k.$  Consider the associated  birational maps $\phi_i : Y\to Y_i,$ where $(Y_i, L_i)$ is the polarized toric variety defined by the polytope $R_i.$ The line bundles $H_i=\phi_i^*(L_i)$ are nef line bundles on $Y$  and the polytopes $P_{(Y,H_i)}$ are affinely isomorphic to $R_i.$
In particular  $[R_0\star\ldots\star R_k]_\pi$ is the polytope defined by  the tautological line bundle  on the toric fibration $\pr(H_0\oplus\ldots\oplus H_k)\to Y.$ Notice that in this case the line bundle $\xi$ is not ample.

If we want to relate $[R_0\star\ldots\star R_k]_\pi$ to a polarized toric fibration we need to enlarge the polytopes $R_i$ is order to get an ample tautological line bundle.
 Consider  the polytopes $P_i=P_{(Y,H_i)}+\sum_0^k R_j.$ The normal fan of $P_i$ is isomorphic to the fan defining the common resolution, $Y,$ for $i=0,\ldots, k.$ Hence the polytopes $P_i$ are normally equivalent. Let $(Y,M_i)$ be the polarized toric variety associated to the polytope $P_i.$  One can then define the Cayley sum $\Cayley(P_0,\ldots,P_k)_{(\pi,Y)},$ whose normal fan is in fact a refinement of the one defining $[R_0\star\ldots\star R_k]_\pi.$ Let $(\pr(M_0\oplus\ldots\oplus M_k)\to Y , \xi)$ be the polarized toric fibration associated to $\Cayley(P_0,\ldots,P_k)_{(\pi,Y)}.$  There is a birational morphism  $ \phi: \pr(M_0\oplus\ldots\oplus M_k) \to X.$ 
  \end{remark}
 \begin{example}
 Consider the polytopes $R_0=\Delta_2, R_1=\Delta_1\times\Delta_1$ in $\QQ^2.$ Consider the  projection onto the first component $\pi:\ZZ^3\to \ZZ$ and  $P=\conv(R_0\times \{0\}, R_1\times \{1\}).$ The polytope $P$ is then isomorphic to $[R_0\star R_1]_\pi,$ and $\pi_\QQ(P)=\Delta_1.$
The common refinement defined by $R_0+R_1$ is the fan of the blow up of $\pr^2$ at two fixed points, $\phi:Y\to \pr^2.$ The polytopes $P_0$ defines the polarized toric variety $(Y,\phi^*({\mathcal O}_{\pr^2}(4))-E_1-E_2)$ and the polytope $P_1$ the pair $(Y,\phi^*({\mathcal O}_{\pr^2}(5))-2E_1-2E_2),$where $E_i$ are the exceptional divisors. The polarized toric fibration $(\pr(M_0\oplus M_1)\to Y , \xi)$ is then $$(\pr([\phi^*({\mathcal O}_{\pr^2}(4))-E_1-E_2]\oplus [\phi^*({\mathcal O}_{\pr^2}(5))-2E_1-2E_2])\to Y, \xi).$$ 
\vspace{.2in}

\begin{tikzpicture}
\draw [fill=gray!50,opacity=.5](0,0,0)--(1,0,0)--(1,0,1)--(0,0,1)--(0,0,0);
\draw [fill=brown!50,opacity=.5](0,1,0)--(1,1,0)--(0,1,1)--(0,1,0);
\draw (1,0,0)--(1,1,0);
\draw (1,0,1)--(1,1,0);
\draw (1,0,1)--(0,1,1);
\draw (0,0,1)--(0,1,1);
\draw [dashed](0,0,0)--(0,1,0);
\node [above] at (0.5,1,0) {$R_1$};
\node [below] at (0.6,0.2,1) {$R_0$};
\node [below] at (0.6,-0.4,1) {$[R_0\star R_1]_\pi $};
\end{tikzpicture}
\begin{tikzpicture}
\draw (0,0)--(2,0)--(2,1)--(1,2)--(0,2)--(0,0);
\draw [fill=brown!50,opacity=.5] (2,0)--(3,0)--(3,1)--(1,3)--(0,3)--(0,2)--(1,2)--(2,1)--(2,0);
\draw [fill=gray!50,opacity=.5] (3,1)--(1,3)--(2,3)--(3,2)--(3,1);
\node [above] at (0.8,1,0) {$R_0+R_1$};
\node [above] at (0.5,2,0) {$P_0$};
\node [below] at (2.2,2.5,0) {$P_1$};
\end{tikzpicture}
\begin{tikzpicture}
\draw [fill=brown!50,opacity=.5] (0,0,0)--(3,0,0)--(3,0,2)--(2,0,3)--(0,0,3)--(0,0,0);
\draw [fill=gray!50,opacity=.5] (0,1,0)--(3,1,0)--(3,1,1)--(1,1,3)--(0,1,3)--(0,1,0);
\draw [dashed](0,0,0)--(0,1,0);
\draw (3,0,0)--(3,1,0);
\draw (0,0,3)--(0,1,3);
\draw (3,0,2)--(3,1,1);
\draw (2,0,3)--(1,1,3);
\node [below] at (2.5,-0.7,1) {$\Cayley(P_0, P_1)_{(\pi,Y)} $};
\end{tikzpicture}
\end{example}

 \vspace{.2in}
 
 \noindent\subsection*{Historical Remark}
 The definition of a Cayley polytope originated from what is  "classically" referred to as the {\em Cayley trick}, in connection with the  Resultant and Discriminant of a system of polynomials.
 A system of $n$ polynomials in $n$ variables $x=(x_1,\ldots,x_n),$  $ f_1(x),\ldots, f_n(x),$  is supported on $(\A_1, \A_2, \ldots, \A_n),$ where $ \A_i\subset\ZZ^n$ if  $f_i=\Pi_{a_j\in \A_i}c_jx^{a_j}.$ 
 
 The $(\A_1, \A_2, \ldots, \A_n)$-{\em resultant} (of ) is a polynomial, $R_\A(\ldots,c_j, \ldots),$ in the coefficients $c_j,$  which vanishes whenever the corresponding polynomials have a common zero.
 
The discriminant of a finite subset $\A\subset\ZZ^n,$ $\Delta_\A,$ is also a polynomial $\Delta_\A(\ldots,c_j, \ldots)$ in the variables  $c_j,$ which vanishes whenever the corresponding polynomial supported on $\A, f=\Pi_{a_j\in \A}c_jx^{a_j},$ has a singularity in the torus $(\CC^*)^n.$
\begin{theorem}\cite{GKZ}[Cayley Trick]
The $(\A_1, \A_2, \ldots, \A_n)$-resultant  equals the $\A$-discriminant where
$$\A=(\A_1\times\{0\})\cup (\A_2\times\{e_1\})\cup\ldots\cup (\A_n\times\{e_{n-1}\})\subset\ZZ^{2n-1}$$
where $(e_1,\ldots,e_{n-1})$ is a lattice basis for $\ZZ^{n-1}.$
\end{theorem}
Let $R_i=N(f_i)\subset\RR^n$ be the Newton polytopes of the polynomials $f_i$ supported on $\A_i.$ The Newton polytope of the polynomial $f$ supported on $\A$ is the Cayley sum 
$$N(f)=[R_1\star\ldots\star R_n]_\pi,$$ where $\pi:\ZZ^{2n-1}\to\ZZ^{n-1}$ is the natural projection such that $\pi_\RR([R_1\star\ldots\star R_n]_\pi)=\Delta_{n-1}.$
 
 \section{Toric discriminants and toric fibrations}\label{Lec2}
 The term  ``discriminant"  is  well known in relation with low degree equations or ordinary differential equations. We will study discriminants of polynomials in $n$ variables with prescribed monomials, i.e. polynomials whose exponents are given by lattice points in $\ZZ^n.$ 

Polynomials in $n$-variables describe locally the hyperplane sections of a projective $n$-dimensional algebraic variety, $\phi:X\hookrightarrow\pr^m.$ The monomials are prescribed by the local representation of a basis of the vector space of  global sections $H^0(X, \phi^*({\mathcal O}_{\pr^m}(1))).$ For this reason the term discriminant has also been classically used in Algebraic Geometry. 

In what follows we will describe discriminants from a combinatorial and an algebraic geometric prospective. The two points of view coincide when the projective embedding is toric.
 
\subsection{\bf The $\A$ discriminant} Let $\A=\{ a_0,\ldots, a_m\}$ be a subset of $\ZZ^n$. The discriminant of $\A$ (when it exists) is an irreducible homogeneous polynomial $\Delta_\A(c_0,\ldots, c_m)$ vanishing when the corresponding Laurent polynomial supported on $\A,$ $f(x)=\sum_{a_i\in{\mathcal A}}c_i x^{a_i},$  has at least one singularity  in the torus $(\CC^*)^n.$ Geometrically, the zero-locus of the discriminant is an irreducible algebraic variety of codimension one in the dual projective space ${\pr^m}^\vee,$ called the {\it dual variety}
of the embedding $X_\A\hookrightarrow\pr^m.$
\begin{example} Consider the point configuration 
$$\A=\{(0,0),(1,0),(0,1),(1,1)\}\subset\ZZ^2.$$ The discriminant is given by an homogeneous  polynomial $\Delta_\A(a_0,a_1,a_2,a_3)$ vanishing whenever the quadric $a_0+a_1x+a_2y+a_3xy$ has a singular point in $(\CC^*)^2.$ It is well known that this locus correspond to singular $2\times 2$ matrices and it is thus described by the vanishing of the determinant:  $\Delta_\A( a_0,a_1,a_2,a_3)=a_0a_3-a_1a_2.$
Similarly, one can associate the polynomials supported on $\A$ with local expansions of global sections in $H^0(\pr^1\times\pr^1, {\mathcal O}(1,1))$ defining the Segre embedding of $\pr^1\times\pr^1$ in $\pr^3.$
\begin{center}
\begin{tikzpicture}
\draw [fill=gray!50,opacity=.5](0,0)--(1,0)--(1,1)--(0,1)--(0,0);
\fill (0,0) circle (3pt) node[above]{} ;
\fill (1,0) circle (3pt) node[above]{} ;
\fill (0,1) circle (3pt) node[above] {};
\fill (1,1) circle (3pt) node[above]{} ;
\node [below] {$\A=\{(0,0),(1,0),(0,1),(1,1)\}\subset\ZZ^2$};
\end{tikzpicture}
\end{center}
\end{example}

\begin{example}\label{veronese}
The $2$-Segre embedding $\nu_2:\pr^2\hookrightarrow \pr^5$ defined by the global sections of the line bundle ${\mathcal O}_{\pr^2}(2)$ can be associated to the point configuration $\A=\{a_0,a_1,a_2,a_3,a_4,a_5\}=\{(0,0),(0,1),(1,0),(0,2),(1,1),(2,0)\}$
\begin{center}
\begin{tikzpicture}
\draw [fill=gray!50,opacity=.5](0,0)--(2,0)--(0,2)--(0,0);
\fill (0,0) circle (3pt) node[above]{} ;
\fill (1,0) circle (3pt) node[above]{} ;
\fill (0,1) circle (3pt) node[above] {};
\fill (1,1) circle (3pt) node[above]{} ;
\fill (0,2) circle (3pt) node[above]{} ;
\fill (2,0) circle (3pt) node[above]{} ;
\end{tikzpicture}
\end{center}
A simple computation shows that $\Delta_\A=\det\left[\begin{matrix} c_0&c_1& c_2\\c_1&c_3&c_4\\
c_2&c_4&c_5  \end{matrix}\right] $
\end{example}

Projective duality is a classical subject in algebraic geometry. Given en embedding $i:X\hookrightarrow \pr^m$ of an $n$-dimensional algebraic variety, the dual variety, $X^\vee\subset(\pr^m)^\vee$  is defined as the Zariski-closure of all the hyperplanes $H\subset\pr^m$ tangent  to $X$ at some non singular point. We can speak of a defining homogeneous polynomial $\Delta(c_0,\ldots,c_m)$, and thus of a  discriminant, only when the dual variety has  codimension one.  Embeddings whose dual variety has higher codimension are called {\it dually defective} and the discriminant is set to be $1.$  
Finding formulas for the discriminant $\Delta$ and giving a classification of  the embeddings  with discriminant $1$  is a long standing problem in algebraic geometry. 
In the case of a toric embedding defined by a point-configuration, $X_\A\hookrightarrow\pr^{|\A|-1},$ the problem is equivalent to finding formulas for the discriminant $\Delta_\A$ and giving a classification of  the {\it dually defective point-configurations,}  i.e. the point-configurations with discriminant $\Delta_\A=1.$   

Dickenstein-Sturmfels, \cite{DS02},  characterized the case when $m=n+2,$ Cattani-Curran, \cite{CC07} extended the classification to $m=n+3,n+4.$ In these cases the corresponding embedding is possibly very singular and the methods used are purely combinatorial. In \cite{DiR06} and \cite{CDR08} we completely characterize the case when $P_\A=\conv({\mathcal A})$ is smooth or  simple. The latter characterization follows from the algebraic geometric characterization, which will be mentioned in the next paragraph.

\subsection{\bf The dual variety of a projective variety} The dual variety corresponds to the locus of singular hyperplane sections of a given embedding. By requiring the singularity to be of a given order, one can define more general dual varieties.
Singularities of fixed multiplicity $k$ correspond to hyperplanes  tangent ``to the order $k.$"  Consider an embedding $i:X\hookrightarrow \pr^m$ of an $n$-dimensional variety, defined by the global sections of the line bundle $\el=i^*({\mathcal O}_{\pr^m}(1)).$ For any smooth point $x$ of the embedded variety let: $$jet^k_x :H^0(X,\el)\to H^0(X, \el\otimes {\mathcal O}_X/{\frak m}_x^{k+1})$$
be the map assigning to  a global section $s$ in $H^0(X,\el)$ the tuple 
$$jet^k_x(s)=(s(x), \ldots, (\partial^t s/\partial\underline{x}^t) (x),\ldots )_{t\leq k}$$
 where $\underline{x}=(x_1,\ldots, x_n)$ is a local system of coordinates around  $x.$
The {\it $k$-th osculating space} at $x$ is defined as $\osc{k}_x=\pr(Im(jet^k_x)).$  As the map $jet^1_x$ is surjective, the first osculating space is always isomorphic to $\pr^n$ and it is classically called the {\it projective tangent space}. The jet maps of higher order do not necessarily have maximal rank and thus the dimension of the osculating spaces of order bigger than $1$ can vary.  The embedding admitting osculating space of maximal dimension at every point are called $k$-jet spanned.

 \begin{definition}
  A line bundle $\el$ on $X$ is called {\it $k$-jet spanned} at $x$ if the map $jet^k_x$ is surjective. It is called $k$-jet spanned  if it is $k$-jet spanned  at every smooth point $x\in X.$
  \end{definition}
  
  \begin{example}
  Consider a line bundle $\el={\mathcal O}_{\pr^n}(a)$ on $\pr^n,$ it is $k$-jet spanned for all  $a\geq k.$ In fact the map
  $$jet_x^k: H^0(\pr^n,{\mathcal O}_{\pr^n}(a))\to J_k({\mathcal O}_{\pr^n}(a))_x$$ is surjective for all $x\in\pr^2,$  as a local basis of the global sections of ${\mathcal O}_{\pr^n}(a)$ consists of all the monomials in $n$ variables of degree up to $a$ and we are assuming $a\geq k.$
  \end{example}
  
  \begin{example}
  Let $\el$ be a line bundle on a non singular toric variety $X.$ Then the following statements are equivalent, see \cite{DiR01} :
\begin{enumerate}
\item $\el$ is $k$-jet spanned
\item all the edges of $P_\el$ have length at least $k.$
\item $\el\cdot C\geq k$ for every invariant curve $C$ on $X.$ 
\end{enumerate}  
As an example consider the polytope $P$ in Figure below. It defines the embedding of the blow up of $\pr^2$ at the three fixpoints,
$\phi:X\to \pr^2,$ defined  by the anticanonical line bundle $\phi^*({\mathcal O}_{\pr^2}(3))-E_1-E_2-E_3.$ Here $E_i$ denote the exceptional divisors. The embedded variety is a Del Pezzo surface of degree $6.$ Let $F$ be the set of the $6$ fixpoints on $X$ and $E=\{\phi^*({\mathcal O}_{\pr^2}(3))-E_i-E_j, i\neq i\}\cup\{E_1,E_2,E_3\}$ be the set of invariant curves.The osculating spaces can easily seen to be:
{\small 
$$\osc{2}_p=\left\{\begin{array}{ccc}
\pr^3&=&<jet_p^2(1),jet_p^2(x),jet_p^2(y),jet_p^2(xy)>\\
&&\text{ if $x\in F.$}\\
\pr^4&=&<jet_p^2(1),jet_p^2(x),jet_p^2(y),jet_p^2(xy),jet_p^2(x^2y)>\\
&&\text{ if $x\in E\setminus F.$}\\
\pr^5&=&<jet_p^2(1),jet_p^2(x),jet_p^2(y),jet_p^2(xy),jet_p^2(x^2y), jet_p^2(xy^2)>\\
&& \text{ at a general point $p\in X\setminus E.$}
\end{array}\right.$$}
\begin{center}
\begin{tikzpicture}
\draw [fill=blue!50,opacity=.2](0,0)--(1,0)--(2,1)--(2,2)--(1,2)--(0,1)--(0,0);
\node [above] at (-0.5,0.5,0) {$P$};
\fill (0,0) circle (3pt) node[above]{} ;
\fill (1,0) circle (3pt) node[above]{} ;
\fill (0,1) circle (3pt) node[above] {};
\fill (1,1) circle (3pt) node[above]{} ;
\fill (2,2) circle (3pt) node[above]{} ;
\fill (2,1) circle (3pt) node[above]{} ;
\fill (1,2) circle (3pt) node[above]{} ;
\end{tikzpicture}
\end{center}
\noindent The embedding defined by $P$ is not $2$-jet spanned on the whole $X.$ It is $2$-jet spanned at every point in $X\setminus E.$

\end{example}

\begin{definition}
A hyperplane $H\subset\pr^m$  is tangent at $x$  to the order $k$ if it contains the $k$-th osculating space at $x$: $\osc{k}_x\subset H.$
\end{definition}
\begin{definition}
The {\it  $k$-th order dual variety} $X^k$ is: 
$$X^k=\overline{\{H\in{\pr^m}^*\text{ tangent  to the order }k\text{ to }X\text{ at some non singular point}\}}.$$  
\end{definition}
Notice that $X^1=X^\vee$ and that $X^2$ is contained in the singular locus of $X^\vee.$
General properties of the  higher order dual variety have been  studied by S. Kleiman and R. Piene. 
Because the definition is related to local osculating properties and generation of jets, it is useful to introduce the sheaf of jets, $J_k(\el),$ associated to a polarized variety $(X,\el).$ In the classical literature it is sometime referred to as the {\it sheaf of principal parts.}

 Consider the projections $\pi_i:X\times X\to X$ and let ${\mathcal I}_{\Delta_X}$ be the ideal sheaf of the diagonal in $X\times X.$ The sheaf of $k$-th order jets  of the line bundle $\el$ is defined as 
 $$J_k(\el)={\pi_2}_*(\pi_1^*(\el)\otimes({\mathcal O}_{X\times X}/{\mathcal I}_{\Delta_X}^{k+1})).$$ When the variety $X$ is smooth $J_k(\el)$ is a vector bundle of rank ${n+k\choose n},$ called the {\it $k$-jet bundle.}
 \begin{example}
 If $\el\neq{\mathcal O}_X$ is a globally generated line bundle then $J_k(\el)$ splits as a sum of line bundles only if $X=\pr^n$ and $\el={\mathcal O}_{\pr^n}(a).$ In fact:
 $$J_k({\mathcal O}_{\pr^n}(a))=\bigoplus_1^{{n+k}\choose k}{\mathcal O}_{\pr^n}(a-k)$$
 See \cite{DRS01} for more details.
 \end{example}
 It is important to note that when the map $jet^k_x$ is surjective for all smooth points $x,$  properties of the higher dual variety $X^k$ can be related to vanishing of Chern classes of the associated $k$-th jet bundle, $J_k(\el).$   
 We start by  identifying  the $k$-th dual variety with a projection of the conormal bundle.  Let $X$ be a smooth algebraic variety and let  $\el$ be a $k$-jet spanned line bundle on $X.$ Consider the following commutative diagram.
 \begin{small}
\begin{equation}\label{diagram}
\xymatrix{ &&&0\dto&\\
&&&S^{k+1}\Omega^1_{X}\otimes\mathcal
\el\dto&\\
&&&J_{k+1}(\el)\dto&\\
0\rto&K_k\uurrto^{II_k^*}\rto|-{\beta_k} &X\times H^0(
X,\el)\urto^{jet^{k+1}}\rto^-{jet^{k}}&
J_{k}(\el)\dto\rto^{}&0\\
&&&0&}
\end{equation}
\end{small}
The vertical exact sequence is often called the $k$-jet sequence. The map $jet^k$ is defined as $jet^k(s,x)=jet^k_x(s).$ The vector bundle $K_k$ is the kernel of the map $jet^k$ (which has maximal rank!).
The induced map $II_k^*$ can be identified with  the dual of the $k$-th fundamental form. See \cite{L94,GH79} for more details. By dualizing the map $\beta_k$ and projectivizing the corresponding vector bundles one gets the following maps:
\begin{small}
\begin{equation}\label{dualcorrespondence}
\xymatrix{
&\pr(K_k^\vee)\ar@{^(->}[d]\ar@/_10pt/[dl]_{\gamma_k}\ar@/^10pt/[dr]^{\alpha_k}\\
X& X\times\pr(H^0(X,L)^\vee)\lto_-{\text{pr}_1}\rto^-{\text{pr}_2}&(\pr^m)^\vee
}
\end{equation}
\end{small}
It is straightforward to see that $X^k=Im(\alpha_k)$. A simple dimension count shows that when the map $jet^k$ has maximal rank one expects the codimension of the $k$-th dual variety to be $\cod(X^k)={{n+k}\choose k}-n.$ Notice that this is equivalent to requiring that the map $\alpha_k$ is generically finite. When the codimension is higher than the expected one  the embedding is said to be {\it $k$-th dually defective.}   

The commutativity in diagram (\ref{diagram}) has the following useful consequence.

\begin{lemma}
Let $(X,\el)$ be a polarized variety, where $X$ is smooth and the line bundle $\el$ is  $(k+1)$-jet spanned. Then the dual variety $X^k$ has  the expected dimension. 
\end{lemma}
\begin{proof}
We follow diagram (\ref{diagram}). Because the line bundle $\el$ is $(k+1)$-jet ample the map $II_k^*$ is surjective. This means  that for every $x\in X$ and for every monomial $ \Pi_{\sum t_i=k+1} x_1^{t_1}\cdots x_n^{t_n}$ there is an hyperplane section that locally around $x$ is  defined as
$$ C\cdot\Pi_{\sum t_i=k+1} x_1^{t_1}\cdots x_n^{t_n}+\text{ higher order terms }=0, \text{ where }C\neq 0$$

In other words,  hyperplanes tangent at a point $x$ to the order $k$ are in one-to-one correspondence with elements of the linear system $|{\mathcal O}_{\pr^{n-1}}(k+1)|.$
The map $\alpha_k$ having positive dimensional fibers is equivalent to saying that hyperplanes tangent at a point $x$ to the order $k$ are also tangent to nearby points $y\neq x,$ which in turn implies that the linear system $|{\mathcal O}_{\pr^{n-1}}(k+1)|$ has base points. This is a contradiction as the linear system is $k+1$-jet spanned  and thus base-point free. \end{proof}

When $k=1$  the contact locus of a general singular hyperplane, $\gamma_k(\alpha_k^{-1}(H))$ is  always a linear subspace. This implies that  if finite then  $\deg(\alpha_1)=1.$ For higher order tangencies, $k>1,$ the degree can be higher. When the map  $\alpha_k$ is finite we  set  $n_k=\deg(\alpha_k).$ 

  \begin{lemma}\cite{LM00, DDRP12}\label{chern1}
  Let $X$ be a smooth variety and let $\el$ be a $k$-jet spanned line bundle. Then $codim(X^k)>{{n+k}\choose k}-n$ if and only if $c_n(J_k(\el))=0.$ Moreover when $codim(X^k)={{n+k}\choose k}-n$ the degree of the $k$-dual variety is given by:  $$n_k\deg(X^k)= c_n((J_k(\el)).$$  \end{lemma}
  \begin{proof}
  Observe first that because the map $jet_k$ is of maximal rank the vector bundle $J_k(\el)$ is spanned by the global sections of the line bundle $\el.$ This implies that, after fixing a basis $\{s_1,\ldots, s_{m+1}\}$ of $H^0(X,\el)\cong\CC^{m+1},$  the Chern class $c_n(J_k(\el))$ is represented by the set:
  $$\{x\in X | \dim(\Span(jet_{x}^k(s_1),jet_x^k(s_2)))\leq 1\}$$
  Notice that an hyperplane in the  linear span $\pr^{t}=\langle s_1,\ldots,s_{t+1}\rangle$ is tangent at a point $x$ to the order $k$ exactly when $\dim(\Span(jet_x^k(s_1),\ldots,jet_x^k(s_{t+1})))=t+1.$
The map $\gamma_k$ in diagram (\ref{dualcorrespondence})  defines a projective bundle  of
 rank $m-{{n+k}\choose k}.$  
 The statement $c_n(J_k(\el))=0$ is then equivalent to  
   $\alpha_k(\gamma ^{-1}(x))\cap\pr^1 =\emptyset$  for every $x\in X$ and for a general
  $\pr^1=\langle s_1,s_2 \rangle.$ By Bertini this is equivalent to   $\cod(X^k)>   {{n+k}\choose k}-n.$
Assume now that $c_n(J_k(\el))\neq 0$ and thus when the generic fiber of the map $\alpha_k$ is finite. The degree of $X^k=im(\alpha_x)$ times the degree of the map $\alpha_k$ is given by the degree of the line bundle $\alpha_k^*({\mathcal O}_{(\pr^m)^\vee}){1}$ which corresponds to the tautological lime bundle ${\mathcal O}_{\pr(K_k^\vee)}(1).$  
$$n_k\deg(X^k)=c_1(\alpha_k^*({\mathcal O}_{(\pr^m)^\vee}{1}))^{m+n-{{n+k} \choose k}}=c_1({\mathcal O}_{\pr(K_k^\vee)}(1))^{m+n-{{n+k} \choose k}}.$$

From diagram (\ref{diagram}) we see that 
$$c_n(J_k(\el))=c_n(K_k^\vee)^{-1}=s_n(K_k^\vee)$$
Finally let $\pi:\pr(K_k^\vee)\to X$ be the bundle map. By relating the Segre class $s_n(K_k^\vee)$ to the tautological bundle, \cite[3.1]{FU},
$s_n(K_k^\vee)=\pi_*(c_1({\mathcal O}_{\pr(K_k^\vee)}(1))^{m+n-{{n+k} \choose k}})=c_1({\mathcal O}_{\pr(K_k^\vee)}(1))^{m+n-{{n+k} \choose k}}$ we conclude that:
$n_k\deg(X^k)=c_n(J_k(\el)).$   \end{proof}
  
  The case of $k=1$ is referred to as classical projective duality. When the codimension of the dual variety is one, the homogeneous polynomial in $m+1$ variables defining it is called the {\it discriminant of the embedding}. For a polarized variety the discriminant, when it exists, parametrizes the singular hyperplane sections.

\subsection{\bf The toric discriminant} In the case of singular  varieties the sheaves of $k$-jets are not necessarily  locally free and thus it is not possible to use  Chern-classes techniques. 

For toric varieties however  estimates of the degree of the dual varieties are possible, even in the singular case, and rely on properties of the associated polytope.  In the classical case $k=1$ there is a precise characterization in any dimension. For higher order duality, results in dimension $3$ and for $k=2$  can be found in  \cite{DDRP12}. A generalization to higher dimension and higher order is an open problem.

  \begin{proposition}\cite{GKZ,DiR06,MT11}
  Let $(X_\A,L_\A)$ be a polarized toric variety associated to the polytope $P_\A.$ Set
  $$\delta_i=\sum_{\emptyset\neq F\prec P} (-1)^{\cod(F)}\{ {\dim(F)+1\choose i}+((-1)^{i-1}(i-1)\}\Vol(F)\Eu(V(F)).$$
  Then $\cod(X_\A^*)=r=min\{i, \delta_i\neq 0\}$ and $\deg(X_\A^*)=\delta_r.$
  \end{proposition}
  
  The function $\Eu: \{\text{invariant subvarieties of }X_A \} \to \ZZ$ in the above proposition assigns an integer to all invariant subvarieties. Its value is different from $1$ only when the variety is singular. In particular, when $X_\A$ is smooth we have:
  \[ \cod(X_A^*)>1 \Leftrightarrow \sum_{\emptyset\neq F\prec P} (-1)^{\cod(F)} (\dim(F)+1)!\Vol(F)=0\]
  In fact in the smooth case one can prove this characterization using the vector bundle of $1$-jets.
  \begin{proposition}\label{chern2} Let $(X_\A,L_\A)$ be an $n$-dimensional  non singular polarized toric variety associated to the polytope $P_\A.$ Assume $\A=P_\A\cap\ZZ^n.$ Then
\[ c_n(J_1(\el_\A))=\sum_{\emptyset\neq F\prec P} (-1)^{\cod(F)} (\dim(F)+1)!\Vol(F)   \] \end{proposition}
\begin{proof}
Chasing the diagram (\ref{diagram}) one sees:
$$ c_n(J_1(\el_\A))=\sum_{i=0}^n(n+1-i)c_i(\Omega^1_{X_\A})\cdot\el_\A^{i}$$
Consider now the generalized Euler sequence for smooth toric varieties, \cite[12.1]{BC94}:
$$0\to \Omega_{X_\A}^1\to \bigoplus_{\xi\in \Sigma_\A(1)}{\mathcal O}_{X_\A}(V(\xi))\to {\mathcal O}_{X_\A}^{|\Sigma_\A(1)|-n}\to 0$$
Where $V(\xi)$ is the invariant divisor associated to the ray $\xi\in \Sigma_\A(1).$ It follows that: $c_i(\Omega_{X_\A}^1))=(-1)^{i} \sum_{\xi_1\neq\xi_2\neq\ldots\neq \xi_i} [V(\xi_1)]\cdot \ldots\cdot [V(\xi_i)].$ Recall that the intersection products $[V(\xi_1)]\cdot \ldots\cdot [V(\xi_i)]$ correspond to codimension $i$ invariant subvarieties and thus faces of $P_\A$ of dimension $n-i.$ Moreover the degree of the embedded subvariety $[V(\xi_1)]\cdot \ldots\cdot [V(\xi_i)]$ is equal to $\el^{n-i}\cdot([V(\xi_1)]\cdot \ldots\cdot [V(\xi_i)])=(n-i)!\Vol(F),$ where $F$ is the corresponding face.
We can then conclude:
$$\begin{array}{c} c_n(J_1(\el_\A))=\sum_{\emptyset\neq F\prec P_\A}(n+1-i)(n-i)!(-1)^i \Vol(F)=\\
=\sum_{\emptyset\neq F\prec P_\A}(-1)^{\cod(F)}(\dim(F)+1)!\Vol(F)\end{array}$$
\end{proof}
\begin{example}
Consider the simplex $2\Delta_2$ in Example \ref{veronese}. All the edges have length  equal to two and therefore the toric embedding is $2$-jet spanned. The dual variety is then an hypersurface and the degree of the discriminant is given by $c_2(J_1({\mathcal O}_{\pr^2}(2))=c_2({\mathcal O}_{\pr^2}(1)\oplus {\mathcal O}_{\pr^2}(1)\oplus {\mathcal O}_{\pr^2}(1))=3.$
The volume formula gives in fact:
$$c_2(J_1({\mathcal O}_{\pr^2}(2))=6Vol(2\delta_2)-2\sum_1^3Vol(2\Delta_1)+3=12-12+3=3.$$
\end{example}
\begin{example} Consider the Segre embedding $\pr^1\times \pr^2\hookrightarrow \pr^5,$ associated to the polytope $Q.$ Then $c_3(J_1(\el))=4!\frac{1}{2}-3!(1+1+1+\frac{1}{2}+\frac{1}{2})+2(9)-6=0.$ This embedding is therefore dually defective.
\begin{center}
\begin{tikzpicture}
\draw [fill=blue!50,opacity=.2](0,0,0)--(1,0,0)--(1,1,0)--(0,1,0)--(0,0,0)--(0,0,1)--(0,1,0)--(0,0,0)--(1,0,0)--(1,0,1)--(0,0,1)--(0,0,0)--(1,0,0)--(1,0,1)--(1,1,0);
\node [above] at (-0.5,0.5,0) {$Q$};
\fill (0,0,0) circle (3pt) node[above]{} ;
\fill (1,0,0) circle (3pt) node[above]{} ;
\fill (0,1,0) circle (3pt) node[above] {};
\fill (1,1,0) circle (3pt) node[above]{} ;
\fill (1,0,1) circle (3pt) node[above]{} ;
\fill (0,0,1) circle (3pt) node[above]{} ;
\end{tikzpicture}
\end{center}
\end{example}
The following is an amusing observation, which is a simple consequence of the previous characterization.
\begin{corollary}\label{eqchern}
Let $P_\A$ be a smooth polytope such that $\A=P_\A\cap\ZZ^n.$ Then
$$\sum_{\emptyset\neq F\prec P_\A}(-1)^{\cod(F)}(\dim(F)+1)!\Vol(F)\geq 0$$
\end{corollary}
\begin{proof} Because the associated line bundle $\el_\A$ defines an embedding of the variety $X_\A,$ the map
$jet^1$ has maximal rank and thus the vector bundle $J_1(\el_\A)$ is spanned (by the global sections of $\el_\A$). It follows that the degree of its Chern classes must be non negative which implies the assertion.
   \end{proof}  
   
   Now we can state the characterization of $\QQ$-factorial and non singular toric embeddings admitting discriminant $\Delta_\A=1.$ The theorem will include the combinatorial characterization and the equivalent algebraic geometry description. The proof in the non singular case will be given in Section \ref{final}. 
   
\begin{theorem}\cite{DiR06,CDR08}\label{characterization1}
Let $\A=P_\A\cap\ZZ^n$ and assume that $X_\A$ is $\QQ$-factorial. Then the following equivalent statements hold.
\begin{enumerate}
\item  The point-configuration $\A$ is dually defective if and only if $P_\A$ is a Cayley sum of the form  $P_\A\cong \Cayley(R_0, \ldots, R_t)_{(\pi,Y)},$ where $\pi(P)$ is a simplex (not necessarily unimodular) in  $\RR^t$ and $R_0,\ldots,R_t$ are normally equivalent polytopes with $t>\frac{n}{2}.$ If moreover $P_\A$ is smooth then $\pi(P)$ is a unimodular simplex.
\item The projective dual variety of the toric embedding $X_\A\hookrightarrow \pr^{|\A\cap\ZZ^n|-1}$ has codimension  $s\geq 2$ if and only if
$X_\A$ is a Mori-fibration, $X_\A\to Y$  and $\dim(Y)< \dim(X)/2.$ If moreover $X_\A$ is non singular then $(X_\A, L_\A)=(\pr(L_0\oplus\cdots\oplus L_t),\xi),$ where $L_i$ are line bundles on a toric variety $Y$ of dimension $m<t.$
\end{enumerate}
  \end{theorem}

   Proposition \ref{characterization1} provides a characterization of the class of smooth polytopes achieving the minimal value $0.$ 
   \begin{corollary}\label{chern}
   Let $P$ be a convex smooth lattice polytope. Then 
   $$\sum_{\emptyset\neq F\prec P_\A}(\dim(F)+1)!(-1)^{\cod(F)}\Vol(F)= 0$$
   If and only if $P=\Cayley(R_0,\ldots,R_t)$ for normally equivalent smooth lattice polytopes $R_i$ with $\dim(R_i)<t.$
   \end{corollary}
  
\section{Toric fibrations and adjunction theory}\label{Lec3}
The classification of projective algebraic varieties is a central problem in Algebraic Geometry dating back to early nineteenth  century. The way one can realistically carry out a classification theory is through invariants, such as the degree, genus, Hilbert polynomial. Modern adjunction theory  and Mori theory are the basis for major advances in this area. 

Let $(X,\el)$ be a polarized $n$-dimensional variety. Assume that $X$ is Gorenstein (i.e. the canonical class $K_X$ is a Cartier divisor). The two key invariants occurring in classification theory, see \cite{Fu90}, are 
the {\em effective log threshold} $\mu(\el)$ and  the {\em nef value} $\tau(\el):$ 
$$\begin{array}{c}\mu(\el) := \sup_\RR \{s \in \QQ \,:\, \dim(H^0(K_X + s \el)) = 0\}\\
\tau(\el) := \min_\RR \{s \in \RR \,:\, K_X + s \el \textup{ is nef}\}.\end{array}$$
Both invariants are at most equal to $n+1.$
Kawamata proved that $\mu(\el)$ is indeed a rational number  and recent advances in the minimal model program establish the same  for $\mu(\el).$  
 They can be visualized as follows.

 Traveling from  $\el$ in the direction of the vector $K_X$ in the Neron-Severi space ${\rm NS}(X)\otimes\RR$ of divisors, $\el+(1/\mu(\el))K_X$ is the meeting point with the cone of effective divisors ${\rm Eff}(X)$ and $\el+(1/\tau(\el))K_X$ is the meeting point with the cone of nef-divisors ${\rm Nef}(X),$ see Fig. \ref{nef-cone}.
 
 A multiple of the nef line bundle $K_X+\tau\el$ defines a morphism $X\to\pr^M$ which can be decomposed (Remmert-Stein factorization) as a composition of a morphism $\phi_\tau:X\to Y$ with connected fibers onto a normal variety $Y$ and finite-to-one morphism $Y\to\pr^M.$ The map $\phi_\tau$ is called the {\it nef-value morphism}.
Kawamata showed that if one writes $r\tau=u/v$ for coprime integers $u,v,$ then:
$$u\leq r(1+\max_{y\in Y}(\dim(\phi_\tau^{-1}(y))) .$$

\begin{corollary}\label{proj}
Let $(X,\el)$ be a polarized  variety. Then the nef-value achieves the maximum value  $\tau(\el)=n+1$ if and only if 
$(X,\el)=(\pr^n, {\mathcal O}_{\pr^n}(1)).$
\end{corollary}
 \begin{proof} Consider the nef value morphism $\phi_\tau:X\to Y$ and observe that 
 $$(n+1)\leq (1+\max_{y\in Y}(\dim(\phi_\tau^{-1}(y))).$$ This implies that the dimension of a fiber of $\phi_\tau$ must be $n$ and thus that the morphism contracts the whole space $X$ to a point.  By construction, the fact that $\phi_\tau$ contracts the whole space implies  that $K_X+(n+1)\el={\mathcal O}_X.$ A celebrated criterion in projective geometry, called the Kobayashi-Ochiai theorem, asserts that if $L$ is an ample line bundle such that $K_X+(n+1)\el={\mathcal O}_X$ then $(X,\el)=(\pr^n, {\mathcal O}_{\pr^n}(1)).$
 \end{proof}

\begin{figure}
\centering
\def\svgwidth{8cm}
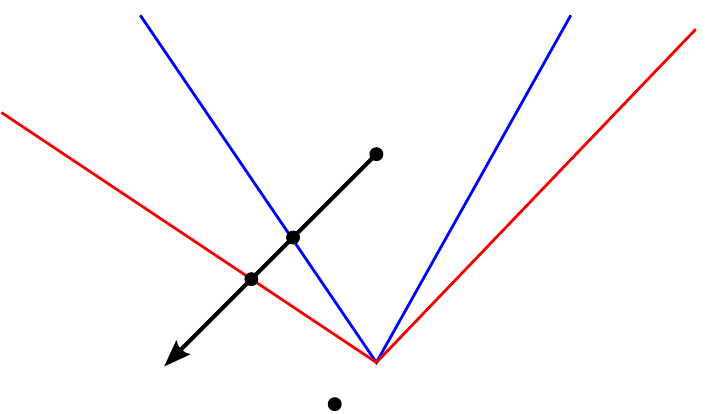
\caption{Illustrating $\mu(\el)$ and $\tau(\el)$}
\label{nef-cone}
\end{figure} 

\begin{remark}\label{remark}Recall that the interior of the closure of the effective cone is the cone of big divisors, $ (\overline{{\rm Eff}(X)})^\circ={\rm Big}(X),$ and that the closure of the ample cone is the nef cone, $\overline{{\rm Ample}(X)}={\rm Nef}(X).$
In particular the equality $\tau(\el)=\mu(\el)$ occurs if and only if the line bundle $K_X+\tau(\el)\el$ is nef and {\it not} big, which implies that $\phi_\tau$ defines a fibration structure on $X.$ 
\end{remark}
 A fibration structure on an algebraic variety is a powerful geometrical tool as many  invariants are induced by corresponding invariants on the (lower dimensional) basis and generic fibre. Criteria for a space to be a fibration are therefore highly desirable. Beltrametti, Sommese and Wisniewski conjectured the if the effective log threshold is strictly  bigger than half the dimension then the nef-value morphism should be a fibration.
\begin{conjecture}\cite{BS94}\label{1}
If $X$ is non singular and   $\mu(\el)>(n+1)/2$ then $\mu(\el)=\tau(\el).$
\end{conjecture}  

Let us now assume that the algebraic variety is toric. 
 In this case it is immediate to see that the defined invariants are rational numbers as the cones ${\rm Eff}(X), {\rm Big}(X), {\rm Ample}(X), {\rm Nef}(X)$ are all rational cones.

We have seen in Section \ref{Lec1} that toric fibrations are associated to certain Cayley polytopes.  Analogously to the classification theory of projective algebraic varieties it is important to find invariants of polytopes that would characterize a Cayley structure. One invariant which has attracted increasing attention in recent years is the {\it codegree} of a lattice polytope:
$$\cd(P)=min\{t\in\ZZ_{>0}\text{ such that } tP\text{ contains interior lattice points} \}.$$
Via Ehrhart theory one  can conclude  that $\cd(P)\leq n+1$ and that $\cd(P)=n+1$ if and only if $P=\Delta_n.$ This is in fact a simple consequence of our previous observations.
\begin{corollary}
Let $P$ be a Gorenstein lattice polytope. Then $\cd(P)=n+1$ if and only if $P=\Delta_n.$
\end{corollary}
\begin{proof}
Let $(X,\el)$ be the Gorenstein toric variety associated to $P.$ Notice that, because $K_X=-\sum D_i$ where the $D_i$ are the invariant divisors, the polytope defined by the line bundle $K_X+t\el$ is the convex hull of the interior points of $tP.$ The equality $\cd(P)=n+1$ is equivalent to  $H^0(K_X+t\el)=0$ for $t\leq n.$ Because nef line bundles must have sections (in particular being nef is equivalent to being  globally generated on toric varieties) we have $\tau(\el)\geq \cd(P)=n+1.$ It follows from Corollary \ref{proj} that $(X,\el)=(\pr^n,{\mathcal O}_{\pr^n}(1))$ and thus $P=\Delta_n.$
\end{proof}

Let us now examine the class of Cayley polytopes we encountered in the characterization of dually defective toric embeddings. We will see  that this is a  class of polytopes  satisfying the strong lower bound
$\cd(P)\geq \frac{\dim(P)}{2}+1$ and the equality $\cd(P)=\mu(\el).$
\begin{lemma}\label{Cayley}
Let $P=\Cayley_{h,Y}(R_0,\ldots,R_t)$ with $t>\frac{n}{2},$ then:
$$\tau(\el_P)=\mu(\el_P)=\cd(P)=t+1\geq \frac{n+3}{2}.$$
\end{lemma}
\begin{proof} Observe that $X_P=\pr(L_0\oplus\ldots\oplus L_t)$ for ample line bundles $L_i$ on the toric variety $Y$ and  $\el=\xi$ is the tautological line bundle. Consider the projective  bundle map $\pi: X_P\to Y.$  The Picard group of $X_P$ is generated by  the pull back of generators of ${\rm Pic}(Y)$ and by the tautological line bundle $\xi.$ Moreover the canonical line bundle is given by the following expression:
$$K_{X_P}=\pi^*(K_Y+L_0+\ldots+L_t)-(t+1)\xi.$$
The toric nefness criterion says that a line bundle on a toric variety is nef if  and only if the intersection with all the invariant curves is non-negative, see for example \cite{ODA}. On the toric variety  $\pr(L_0\oplus\ldots\oplus L_t)$ there are two types of rational invariant curves. The ones contained in the fibers $F\cong \pr^t$ and the pull back of rational invariant curves in $Y$ which will be denoted by $\pi^*(C)_i$ when contained in the invariant section defined by the polytope $R_i.$ For any rational invariant curve $C\subset F,$ it holds that  $\xi|_C={\mathcal O}_{\pr^1}(1)$ and  $\pi^*(D)\cdot C=0,$ for all divisors $D$ on $Y.$   For every  curve of the form $\pi^*(C)_i$  it holds that  $\pi^*(C)_i\cdot \pi^*(D)=C\cdot D$ and $\xi\cdot \pi^*(C)_i=L_i\cdot C.$ See \cite[Remark 3]{DiR06} for more details.  We conclude that $K_{X_P}+s\el$ is nef if the following is satisfied:
$$\begin{array}{cc}
[\pi^*(K_Y+L_0+\ldots+L_t)+(s-t-1)\xi]C=s-t-1\geq 0& \text{ if }C\subset F\\
(K_Y+L_0+\ldots+(s-t)L_i+\ldots +L_t)\cdot C\geq 0& \text{ if }C=\pi^*(C)_i\end{array}
$$
In \cite{MU02} Mustata proved a toric-Fujita conjecture showing that if for a line bundle $H$ on an $n$-dimensional toric variety, $H\cdot C\geq n$ for every invariant curve $C,$ then the adjoint bundle $K+H$ is globally generated, unless $H={\mathcal O}_{\pr^n}(n).$
Because $$[L_0+\ldots+(s-t)L_i+\ldots +L_t]\cdot C\geq (s-t)+t$$  it follows that $(K_Y+L_0+\ldots+(s-t)L_i+\ldots +L_t)\cdot C\geq 0$ for all invariant curves $C=\pi^*(C)_i$ if $s\geq t.$ This implies that $K_{X_P}+s\el$ is nef if and only if $s\geq t+1$ and thus  $\tau(\xi)=t+1.$

Consider now the projection $h:\RR^{n}\to\RR^t$ such that $h(P)=\Delta_t.$ Under this projection interior points of a  dilation $sP$ are mapped to  interior points of the corresponding dilation $s\Delta_t.$ This implies that $\cd(P)= t+1.$ Notice that $\mu(\el)\leq\cd(P)=t+1$ as interior points of $sP$ correspond to global sections of $K_{X_P}+s\el.$ 
On the other hand, see \cite[Ex. 8.4]{HA}:
$$\begin{array}{c}H^0(u(\pi^*(K_Y+L_0+\ldots+L_t))+(v-u(t+1))\xi)=\\ =H^0(\pi_*(u(\pi^*(K_Y+L_0+\ldots+L_t))+(v-u(t+1))\xi))= \\ =H^0(u(K_Y+L_0+\ldots+L_t)+\pi_*((u-v(t+1))\xi)=0\text{ if } v-u(t+1)< 0.\end{array}$$
This implies that $\mu(\el)\geq t+1$, which proves the assertion.
\end{proof}
Recently  Batyrev and Nill in \cite{BN08} classified polytopes with $\cd(P)=n$ and conjectured the following. \begin{conjecture}\cite{BN08} \label{2}
There is a function $f(n)$ such that  any  $n$-dimensional polytope $P$ with $\cd(P)\geq f(n)$  decomposes as a Cayley sum of lattice polytopes.\end{conjecture}
The above conjecture was proven by Haase, Nill and Payne in \cite{HNP09}. They showed  that $f(n)$ is at most quadratic in $n.$ 
It is important to observe that, as interior lattice points of $tP$ correspond to global sections of $K_X+t\el$ for the associated toric embedding, $\cd(P)$ can be considered as  the {\it integral variant} of $\mu(\el).$
This  observation, techniques from  toric Mori theory and adjunction theory led   to  prove a stronger version of Conjectures \ref{1} and \ref{2}  for smooth polytopes giving yet another characterization of Cayley sums. 
\begin{theorem}\cite{DDRP09, DN10}\label{conj} Let $P\subset\RR^n$ be a smooth $n$-dimensional  polytope.  Then the following statements are equivalent.
\begin{enumerate}
\item $\cd(P)\geq (n+3)/2.$
\item $P$ is affinely isomorphic to a Cayley sum $\Cayley(R_0,\ldots,R_t)_{\pi,Y}$ where $t+1=\cd(P)$ with $t>\frac{n}{2}.$
\item $\mu(\el_P)=\tau(\el_P)=t+1$ and $t>\frac{n}{2}.$
\item $(X_P, \el_P)=(\pr(L_0\oplus\cdots\oplus L_t),\xi)$ for ample line bundles $L_i$ on a non singular toric variety $Y.$
\end{enumerate}
 \end{theorem}  
Notice that  Theorem \ref{conj} proves the reverse statement of Lemma \ref{Cayley}.

Conjectures \ref{1} and \ref{2}, made independently in two apparently unrelated fields, constitute a beautiful example of the interplay between classical projective (toric) geometry and convex geometry.
In view of the results above one could hope that in the toric  setting the conjectures should hold in more generally. 

\begin{conjecture}\label{A}
Let $(X, \el)$ be an $n$-dimensional toric polarized variety  (not necessarily  smooth or even Gorenstein), then
$\mu(\el)>(n+1)/2$ implies that $\mu(\el)=\tau(\el).$
\end{conjecture}
The invariants   $\mu(\el), \tau(\el)$ in the non Gorenstein case can be defined using corresponding invariants, $\mu(P), \tau(P)$ of the associated polytope, see below for a definition.
\begin{conjecture}\label{B}
  If an $n$-dimensional lattice polytope $P$ satisfies $\cd(P) >
  (n+2)/2$, then it decomposes as a Cayley sum of lattice
  polytopes of dimension at most $2 (n+1-\cd(P))$.
\end{conjecture} 
Conjecture \ref{A} is a toric version of Conjecture \ref{1}, extending the statement to possibly singular and non Gorenstein varieties. Conjecture \ref{B} states that the  function $f(n)$ in Conjecture \ref{2}   should be equal to $(n+2)/2.$
An important step to prove these conjectures is to define the convex analog of $\mu(\el_P).$ 

Let $P \subseteq \RR^n$ be a rational polytope of dimension $n$.
Any such polytope $P$ can be described in a unique minimal way as
\[P = \{x \in \RR^n \,:\,  \pro{a_i}{x} \geq b_i,\; i=1, \ldots, m\}\]
where the $a_i$ are the rows of an $m \times n$ integer matrix $A$, and $b \in \QQ^m$.

For any $s \geq 0$ we define the {\em adjoint polytope} $\adP{s}$ as
\[\adP{s} := \{x \in \RR^n \,:\, A x \geq b+s{\bf 1}\},\]
where ${\bf 1 }= (1, \ldots, 1)^\transpose$. 

We call the study of such polytopes $\adP{s}$ {\em polyhedral adjunction theory}. 
\begin{figure}[ht]
\includegraphics[height=2cm]
{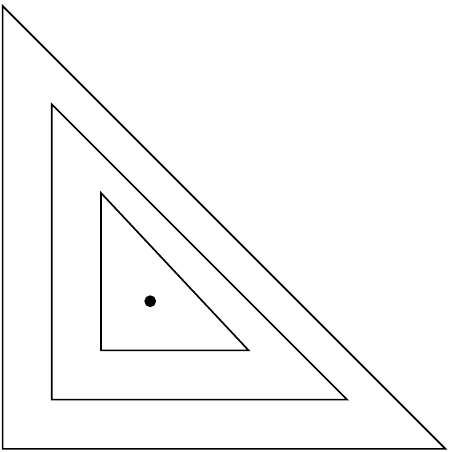}
\hspace{2cm}
\includegraphics[height=2cm]
{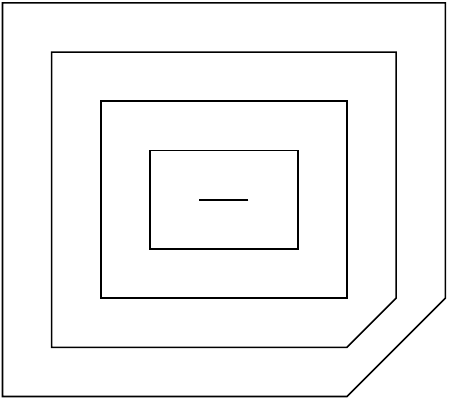}
\caption{Two examples of polyhedral adjunction}
\label{fig:start}
\end{figure}

\begin{definition}\label{def:setting}
We define the {\em $\QQ$-codegree} of $P$ as 
\[\mu(P) := (\sup\{s > 0 \,:\, \ad{P}{s} \not=\emptyset\})^{-1},\]
and the {\em core} of $P$ to be $\core(P) := \adP{1/\mu(P)}$.
\end{definition}

Notice that  in this case the supremum is actually a maximum. Moreover, 
since $P$ is a rational polytope, $\mu(P)$ is a positive rational number.

One sees that  for a \emph{lattice} polytope $P$
\[\mu(P) \leq \cd(P) \leq n+1\]

\begin{definition} 
  The {\em nef value} of $P$ is given as
  \[\nf(P) := (\sup\{s > 0 \,:\, \NF(\ad{P}{s}) = \NF(P)\})^{-1} \in
  \RR_{>0} \cup \{\infty\}\]
  where $\NF(P)$ denotes the normal fan of the polytope $P.$
\end{definition}

Note that in contrast to the definition of the $\QQ$-codegree, here 
the supremum is never a maximum. 

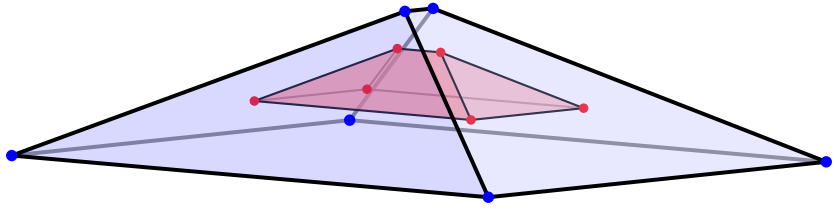
\begin{figure}[ht]
  \centering
    \begin{tikzpicture}[perspective adjusted,scale=.18]
      
    \coordinate (a) at (  0,   1,   1);%
    \coordinate (b) at (  0,   1,  -1);%
    \coordinate (c) at ( 22, -10,  12);%
    \coordinate (d) at ( 22, -10, -12);%
    \coordinate (e) at (-22, -10, -12);%
    \coordinate (f) at (-22, -10,  12);%

    \coordinate (g) at ( -2, -2,  0);%
    \coordinate (h) at (  2, -2,  0);%
    \coordinate (i) at ( 10, -6, -4);%
    \coordinate (j) at ( 10, -6,  4);%
    \coordinate (k) at (-10, -6, -4);%
    \coordinate (l) at (-10, -6,  4);%

    \draw [line width=1.5pt,color=gray] (d) -- (e);%
    \draw [line width=1.5pt,color=gray] (b) -- (e);%
    \draw [line width=1.5pt,color=gray] (f) -- (e);%

    \draw [thick,color=black!70] (g) -- (k);%
    \draw [thick,color=black!70] (i) -- (k);%
    \draw [thick,color=black!70] (k) -- (l);%

    \fill [fill=red!50,opacity=.7] (g) -- (h) -- (j) -- (l);
    \fill [fill=red!30,opacity=.7] (h) -- (i) -- (j);

    \draw [thick] (g) -- (h);%
    \draw [thick] (h) -- (i);%
    \draw [thick] (h) -- (j);%
    \draw [thick] (i) -- (j);%
    \draw [thick] (g) -- (l);%
    \draw [thick] (j) -- (l);%
    \fill (g) [red] circle (10pt);%
    \fill (h) [red] circle (10pt);%
    \fill (i) [red] circle (10pt);%
    \fill (j) [red] circle (10pt);%
    \fill (k) [red] circle (10pt);%
    \fill (l) [red] circle (10pt);%

    \fill [fill=blue!50,opacity=.3] (a) -- (c) -- (f);
    \fill [fill=blue!30,opacity=.3] (a) -- (b) -- (d) -- (c);

    \draw [line width=1.5pt] (a) -- (b);%
    \draw [line width=1.5pt] (a) -- (c);%
    \draw [line width=1.5pt] (b) -- (d);%
    \draw [line width=1.5pt] (a) -- (f);%
    \draw [line width=1.5pt] (c) -- (f);%
    \draw [line width=1.5pt] (d) -- (c);%
    \fill (a) [blue] circle (12pt);%
    \fill (b) [blue] circle (12pt);%
    \fill (c) [blue] circle (12pt);%
    \fill (d) [blue] circle (12pt);%
    \fill (e) [blue] circle (12pt);%
    \fill (f) [blue] circle (12pt);%
  \end{tikzpicture}
  \caption{$\adP 4 \subseteq P$ for a $3$-dimensional lattice polytope $P$}
  \label{fig:pyramid}
\end{figure}
Fig. 4 illustrates a polytope $P$ with $\nf(P)^{-1} = 2$ and $\mu(P)^{-1} = 6.$ In this case $\core(P)$ is an interval.

 In   \cite{DRHNP13}   the precise analogue of Conjecture \ref{B} for the $\QQ$-codegree is proven.  
\begin{theorem}\cite{DRHNP13} Let $P$ be an $n$-dimensional lattice polytope. If $n$ is odd and
  $\mu(P) > (n+1)/2$, or if $n$ is even and $\mu(P) \geq
 (n+1)/2$, then $P$ is a Cayley polytope.\end{theorem}

  Results from \cite{DRHNP13}  show Conjecture \ref{B}  in two interesting cases: when $\up{\mu(P)} = \cd(P)$ and when the normal fan of $P$ is Gorenstein and  $\mu(P)=\tau(P).$  
  
  \section{Connecting the three characterizations }\label{final}
  
   In Section \ref{Lec2} we have seen that a certain class of Cayley polytopes characterizes dually defective configuration points. Moreover this class corresponds to the polytopes achieving the equality in Corollary \ref{eqchern}. In Section \ref{Lec3} the same class of Cayley polytopes was  characterized as corresponding to smooth configurations with codegree larger than slightly more that half the dimension. We will here assemble the three characterizations  and provide  proofs in the non singular case.
  
  \begin{theorem}\label{final theorem}
  Let $\A\subset\ZZ^n$ be a point configuration such that $P_\A\cap \ZZ^n=\A,$ $\dim(P_\A)=n$ and such that $P_\A$ is a smooth polytope. Then the following statements are equivalent.
\begin{enumerate}
\item $P_\A$ is affinely isomorphic to a Cayley sum $\Cayley(R_0,\ldots,R_t)_{\pi,Y}$ where $t+1=\cd(P_\A)$ and $t>\frac{n}{2}.$
\item $\cd(P_\A)>\frac{n+1}{2}+1$ and $\tau(P_\A)=\mu(P_\A).$
\item The discriminant $\Delta_\A=1.$
\item $\sum_{\emptyset\neq F\prec P_\A}(\dim(F)+1)!(-1)^{\cod(F)}\Vol(F)= 0$
\end{enumerate}
  \end{theorem}
\begin{proof}
\noindent [{\bf $(d) \Leftrightarrow (c)$}]. The implication $(d)\leadsto (c)$  follows from Lemma \ref{chern1} and Proposition \ref{chern2}. The reverse implication follows from Corollary \ref{chern}. 

\noindent [$(c)\Rightarrow (b).$] Assume now $(c),$ i.e. assume that the configuration is dually defective. Consider the associated polarized toric manifold $(X_\A,\el_\A).$ It is a classical result that the generic tangent hyperplane is in fact tangent along  a linear space in $X_\A.$ Therefore if $\cod(X_\A^\vee)=k>1$ then there is a linear $\pr^k$ through a general point of $X_\A.$ By linear $\pr^k$ we mean a subspace $Y\cong \pr^k$ such that $L|_Y={\mathcal O}_{\pr^k}(1).$ Moreover, by a result of Ein, \cite{E86},  $N_{\pr^{k}/X}={\mathcal O}_{\pr^k}((-n-2-k)/2).$ Observe that if we fix a point $x\in X_\A,$ a sequence $\{F_j\}$ of general linear subspaces $F_J\cong\pr^k$ can be chosen so that $x\in\lim(F_j).$ Since the $F_i$ are all linear the limit space has to be also a linear $\pr^k.$ We can then assume that there is a linear $\pr^k$ through every  point of $X_\A.$ Let $L$ now be an invariant line in one of the $\pr^k$ through a fixed point. Then:
$$[K_X+t\el ]_L={\mathcal O}_{\pr^1}( (-n-2-k)/2+t)$$
which implies $\tau(\el_\A)\geq \frac{n+k}{2}+1.$
Assume now that $\tau(\el_\A)> \frac{n+k}{2}+1$  and let  $L$ be again a line in the family of linear spaces covering $X.$ The quantity $-K_X\cdot L-2=\nu$ is called the normal degree of the family. In our case  $\nu=(n+k)/2 -1>n/2.$  By a result of Beltrametti-Sommese-Wisniewski, \cite{BSW92}, this assumption implies  $\nu=\tau-2,$ proving  $\tau(\el_A)= \frac{n+k}{2}+1.$ 
Notice that the nef-morphism $\phi_\tau$ contracts all the linear $\pr^k$ of the covering family  and thus it is a fibration. As a consequence the line bundle $K_{X_\A}+\tau\el_\A$ is not big and thus $\tau(\el_\A)=\mu(\el_\A).$ The inequality
$$\cd(P_\A)\geq\mu(\el_{\A})=\tau (\el_\A)= \frac{n+k}{2}+1$$ shows the implication $(c)\leadsto (b).$ 

\noindent $[(b)\Rightarrow (a)].$ Assume now $(b)$. The nef-value morphism is then a fibration and $$\tau(\el_\A)> \cd(P_\A)-1>\frac{n}{2}.$$ Notice that the nef-morphism $\phi_\tau$ contracts a face of the 
Mori-cone and thus faces of the lattice polytope  $P_\A,$ i.e. all the invariant curves with $0$-intersection with the line bundle $K_{X_\A}+\tau\el_\A.$  Let now $C$ be a generator of an 
extremal ray contracted by the morphism $\phi_\tau.$ If $\el_\A\cdot C\geq 2$, then $-K_X\cdot C> n+1$ which is impossible. We can  conclude that $C$ is a line and 
$\tau(\el_\A) =-K_{X_\A}\cdot C$ is an integer. It follows that $\tau (\el_A) \geq \frac{n+1}{2}+1.$ This inequality   implies that $\phi_\tau$ is the contraction of one extremal ray, by 
\cite[Cor. 2.5]{BSW92}. These morphisms are analyzed in detail in \cite{Re83}. Because $X_\A$ is smooth and toric and this contraction has connected fibres, the general fiber $F$ of the contraction  is a smooth toric variety with Picard number one. It follows that $F$ is a projective space and thus
 $\phi_\tau$ is a $\pr^t$ bundle. Let  $L|_F={\mathcal O}_{\pr^t}(a).$ Observe that by construction  $K_{X_\A}|_F=K_F.$ Consider a line $l\subset F.$  It follows that 
$$0=(K_{X_\A}+\tau \el_\A)\cdot l=K_F\cdot l+\tau\el_\A\cdot l=-t-1+a\tau$$
and thus  $\tau=\frac{t+1}{a}> \frac{n+1}{2}+1$ which implies $a=1$ and $t> \frac{n+1}{2}.$
Since $a=1$ the fibers are embedded linearly and thus   $(X_\A,\el_\A)=(\pr(L_0\oplus\ldots\oplus L_t),\xi),$ for ample line bundles $L_i$ on a smooth toric variety $Y.$ This proves the implication $(b)\leadsto (a).$

\noindent $[(a)\Rightarrow (c)].$ Assume now $(a)$. Using notation as in (\ref{dualcorrespondence}), consider the commutative diagram:

\begin{small}
\[\xymatrix{ &\alpha^{-1}(\pr(E)^\vee_{reg})\drto^{\alpha}\dlto_{\gamma}&\\
\pr(E)\drto^{\pi}&&\pr(E)^\vee_{reg}\dlto^{f}\\
&Y&
}\]
\end{small}
where $E=L_0\oplus\ldots \oplus L_t$ and $(Y,L_i)$ is the smooth polarized variety associated to the polytope $R_i.$ The commutativity of the diagram and the existence of $f$ follows from \cite[Lemma 1.15]{DeB01}. Let $y\in Y$ and let $F\cong\pr^t\subset \pr^{|\A|-1}$ be the fiber $\pi^{-1}(y).$ Commutativity of the diagram implies that  the contact locus $\gamma(\alpha^{-1}(H))$ is included in $F$ for all $H\in f^{-1}(y)$. Moreover $\osc{}_{F,y}\subset \osc{}_{\pr(E),y}\subset H$ implies that $H$ belongs to the dual variety $F^\vee,$ with contact locus at least of the same dimension. Because the map $f$ is dominant we can conclude that:
$\dim(F^\vee)\geq \dim(\pr(E)^\vee)-\dim(Y),$ which implies 
$$\cod(\pr(E)^\vee)\geq \cod(F^\vee)-\dim(Y).$$ Recall that the fibers $F$ are embedded linearly and thus $\cod(F^\vee)=\dim(F)+1.$ It follows that
$\cod(\pr(E)^\vee)\geq \dim(F)+1 -\dim(Y)>1$ and thus $\Delta_\A=1$. This proves $(a)\leadsto (c).$

\end{proof}


\raggedright


\begin{bibdiv}
\begin{biblist}

\bib{BC94}{article}{
label={BC94},
author={V.V.~Batyrev },
author={D.~Cox},
title={ On the Hodge structures of projective hyper surfaces in toric varieties},
journal={Duke Math.}
  VOLUME = {75},
      YEAR = {1994},
    NUMBER = {},
     PAGES = {293--338},}


\bib{BN08}{article}{
label={BN08},
author={V.V.~Batyrev },
author={B.~Nill},
title={ Integer points in polyhedra},
journal={In Matthias Beck and et. al., editors, {\em Combinatorial aspects of mirror symmetry}, volume 452 of {\em Contemp. Math.}, pp.~35--66. AMS, 2008.},}

\bib{BS94}{article}{
label={BS94},
author={M.C. Beltrametti },
author={A.J.~Sommese},
title={Some effects of the spectral values on reductions.},
journal={{\em Classification of algebraic varieties ({L}'{A}quila, 1992)},
  volume 162 of {\em Contemp. Math.}, pp.~31--48. Amer. Math. Soc.,
  Providence, RI, 1994.},}

\bib{BSW92}{incollection}{
label={BSW92},
    author={M.C. Beltrametti },
author={A.J.~Sommese},
author={J.A. Wisniewski},
     TITLE = {Results on varieties with many lines and their applications to
              adjunction theory},
 BOOKTITLE = {Complex algebraic varieties ({B}ayreuth, 1990)},
    SERIES = {Lecture Notes in Math.},
    VOLUME = {1507},
     PAGES = {16--38},
 PUBLISHER = {Springer},
   ADDRESS = {Berlin},
      YEAR = {1992},
   MRCLASS = {14C20 (14J40)},
  MRNUMBER = {1178717 (93i:14007)},
MRREVIEWER = {Antonio Lanteri},
       DOI = {10.1007/BFb0094508},
       URL = {http://dx.doi.org/10.1007/BFb0094508},
}

  
 

\bib{CC07}{article}{
label={CC07},
    author={  E. Cattani},
    AUTHOR = {R. Curran},
     TITLE = {Restriction of {$A$}-discriminants and dual defect toric
              varieties},
   JOURNAL = {J. Symbolic Comput.},
  FJOURNAL = {Journal of Symbolic Computation},
    VOLUME = {42},
      YEAR = {2007},
    NUMBER = {1-2},
     PAGES = {115--135},
      ISSN = {0747-7171},
   MRCLASS = {14M25},
  MRNUMBER = {2284288 (2007j:14076)},
MRREVIEWER = {Joseph P. Rusinko},
       DOI = {10.1016/j.jsc.2006.02.006},
       URL = {http://dx.doi.org/10.1016/j.jsc.2006.02.006},
}


 \bib{Mix12}{article}{
label={CCDDRS},
author={E. Cattani},
author={M.A. Cueto},
author={A. Dickenstein},
author={S. Di Rocco},
author={B. Sturmfels},
 title={Mixed discriminants},
 journal={ Math Z.}
 YEAR = {2013},
    NUMBER = {274},
     PAGES = {761-778},}
 
 \bib{CDR08}{article}{
label={CDR08},
    AUTHOR = {C. Casagrande},
    author={S. Di Rocco},
     TITLE = {Projective {$\Bbb Q$}-factorial toric varieties covered by
              lines},
   JOURNAL = {Commun. Contemp. Math.},
  FJOURNAL = {Communications in Contemporary Mathematics},
    VOLUME = {10},
      YEAR = {2008},
    NUMBER = {3},
     PAGES = {363--389},
}

\bib{DeB01}{book}{
label={DeB01},
    AUTHOR = {O. Debarre},
     TITLE = {Higher-dimensional algebraic geometry},
    SERIES = {Universitext},
 PUBLISHER = {Springer-Verlag},
   ADDRESS = {New York},
      YEAR = {2001},
     PAGES = {xiv+233},
      ISBN = {0-387-95227-6},
   MRCLASS = {14-02 (14E30 14Jxx)},
  MRNUMBER = {1841091 (2002g:14001)},
MRREVIEWER = {Mark Gross},
}

 
 \bib{DDRP12}{article}{
label={DDRP12},
author={A. Dickenstein},
author={S. Di Rocco},
author={R. Piene},
 title={Higher order duality and toric embeddings},
 journal={to appear in Annales de l'Institut Fourier }}
 
 \bib{DFS07}{article}{
label={DFS07},
    AUTHOR = {A. Dickenstein},
    author={ E. M. Feichtner},
    Autor={B. Sturmfels},
     TITLE = {Tropical discriminants},
   JOURNAL = {J. Amer. Math. Soc.},
  FJOURNAL = {Journal of the American Mathematical Society},
    VOLUME = {20},
      YEAR = {2007},
    NUMBER = {4},
     PAGES = {1111--1133},
}


\bib{DN10}{article}{
label={DN10},
author={A. Dickenstein},
author={B. Nill},
title={A simple combinatorial criterion for projective toric manifolds with dual defect} 
journal={{\emph Mathematical Research Letters}, 17: 435-448, 2010.},}


\bib{DS02}{article}{
label={DS02},
    AUTHOR = {A. Dickenstein},
    author={ B. Sturmfels},
     TITLE = {Elimination theory in codimension 2},
   JOURNAL = {J. Symbolic Comput.},
  FJOURNAL = {Journal of Symbolic Computation},
    VOLUME = {34},
      YEAR = {2002},
    NUMBER = {2},
     PAGES = {119--135},
      ISSN = {0747-7171},
   MRCLASS = {14M25 (13P05 52B55 68W30)},
  MRNUMBER = {1930829 (2003h:14073)},
MRREVIEWER = {Margherita Barile},
       DOI = {10.1006/jsco.2002.0545},
       URL = {http://dx.doi.org/10.1006/jsco.2002.0545},
}



\bib{DS02}{article}{
label={DS02}
    AUTHOR = {A. Dickenstein},
    author={B. Sturmfels},
     TITLE = {Elimination theory in codimension 2},
   JOURNAL = {J. Symbolic Comput.},
  FJOURNAL = {Journal of Symbolic Computation},
    VOLUME = {34},
      YEAR = {2002},
    NUMBER = {2},
     PAGES = {119--135},
}

 \bib{DDRP12}{article}{
label={DDRP12},
author={A. Dickenstein},
author={S. Di Rocco},
author={R. Piene},
 title={Higher order duality and toric embeddings},
 journal={to appear in Annales de l'Institut Fourier }}

\bib{DDRP09}{article}{
label={DDRP09},
    AUTHOR = {A. Dickenstein},
   author={ S. Di Rocco},
    author={R. Piene},
     TITLE = {Classifying smooth lattice polytopes via toric fibrations},
   JOURNAL = {Adv. Math.},
  FJOURNAL = {Advances in Mathematics},
    VOLUME = {222},
      YEAR = {2009},
    NUMBER = {1},
     PAGES = {240--254},
}

\bib{DiR06}{article}{
label={DiR06}
    AUTHOR = {S. Di Rocco},
     TITLE = {Projective duality of toric manifolds and defect polytopes},
   JOURNAL = {Proc. London Math. Soc. (3)},
  FJOURNAL = {Proceedings of the London Mathematical Society. Third Series},
    VOLUME = {93},
      YEAR = {2006},
    NUMBER = {1},
     PAGES = {85--104},
}	
\bib{DiR01}{article}{
label={DiR01}
    AUTHOR = {S. Di Rocco},
     TITLE = {Generation of k-jets on toric varieties},
   JOURNAL = {Math. Z.},
  FJOURNAL = {},
    VOLUME = {231},
      YEAR = {1999},
    NUMBER = {},
     PAGES = {169--188},}

\bib{DRHNP13}{article}{
label={DRHNP13},
author={S. Di Rocco},
author={ B. Nill},
author={C. Haase},
author={A. Paffenholz},
title={Polyhedral Adjunction Theory},
journal={Algebra and Number Theory},
YEAR = {to appear},
    NUMBER = {}
pages={},
}

\bib{DRS04}{article}{
label={DRS04},
author={S. Di Rocco},
author={ A. J. Sommese},
title={Chern numbers of ample vector bundles on toric surfaces},
journal={Trans. Amer. Math. Soc.},
YEAR = {2004},
    NUMBER = {356, 2}
pages={587--598},
}

\bib{DRS01}{article}{
label={DRS01},
author={S. Di Rocco},
author={ A. J. Sommese},
title={Line bundles for which a projectivized jet bundle is a product},
journal={Proc. of A.M.S.},
YEAR = {2001},
    NUMBER = {129, 6}
pages={1659--1663},
}

\bib{E86}{article}{
label={E86},
author={L. Ein},
title={Varieties with small dual varieties},
journal={Inv. Math},
YEAR = {1986},
    NUMBER = {96}
pages={63--74},
}

\bib{EW}{book}{
label={EW},
    AUTHOR = {G. Ewald},
     TITLE = {Combinatorial convexity and algebraic geometry},
    SERIES = {Graduate Texts in Mathematics},
    VOLUME = {168},
 PUBLISHER = {Springer-Verlag},
   ADDRESS = {New York},
      YEAR = {1996},
     PAGES = {xiv+372},
}

 \bib{Fu90}{book}{
   label={Fuj90},
    AUTHOR = {T. Fujita},
     TITLE = {Classification Theories of Polarized Varieties},
    SERIES = {Lecture Note Ser.},
 PUBLISHER = {London Math. Soc,Cambridge Univ. Press, 155 },
         YEAR = {1990},}
         
         \bib{FUb}{book}{
label={FUb},
    AUTHOR = {W. Fulton},
     TITLE = {Intersection theory},
    SERIES = {Ergebnisse der Mathematik und ihrer Grenzgebiete. 3. Folge. A
              Series of Modern Surveys in Mathematics [Results in
              Mathematics and Related Areas. 3rd Series. A Series of Modern
              Surveys in Mathematics]},
    VOLUME = {2},
   EDITION = {Second},
 PUBLISHER = {Springer-Verlag},
   ADDRESS = {Berlin},
      YEAR = {1998},
     PAGES = {xiv+470},
      ISBN = {3-540-62046-X; 0-387-98549-2},
   MRCLASS = {14C17 (14-02)},
  MRNUMBER = {1644323 (99d:14003)},
       DOI = {10.1007/978-1-4612-1700-8},
       URL = {http://dx.doi.org/10.1007/978-1-4612-1700-8},
}

\bib{FU}{book}{
label={FU},
    AUTHOR = {W. Fulton},
     TITLE = {Introduction to toric varieties},
    SERIES = {Annals of Mathematics Studies},
    VOLUME = {131},
      NOTE = {The William H. Roever Lectures in Geometry},
 PUBLISHER = {Princeton University Press},
   ADDRESS = {Princeton, NJ},
      YEAR = {1993},
     PAGES = {xii+157},
      ISBN = {0-691-00049-2},
   MRCLASS = {14M25 (14-02 14J30)},
  MRNUMBER = {1234037 (94g:14028)},
MRREVIEWER = {T. Oda},
}


\bib{GKZ}{book}{
label={GKZ},
    AUTHOR = {I.M. Gel{\cprime}fand},
    author={ M.M. Kapranov},
    author={A. Zelevinsky},
         TITLE = {Discriminants, resultants, and multidimensional determinants},
    SERIES = {Mathematics: Theory \& Applications},
 PUBLISHER = {Birkh\"auser Boston Inc.},
   ADDRESS = {Boston, MA},
      YEAR = {1994},
     PAGES = {x+523},
      ISBN = {0-8176-3660-9},
   MRCLASS = {14N05 (13D25 14M25 15A69 33C70 52B20)},
  MRNUMBER = {1264417 (95e:14045)},
MRREVIEWER = {I. Dolgachev},
       DOI = {10.1007/978-0-8176-4771-1},
       URL = {http://dx.doi.org/10.1007/978-0-8176-4771-1},
}

\bib{GH79}{article}{
label={GH79},
    AUTHOR = {P. Griffiths}, 
    author={J. Harris},
     TITLE = {Algebraic geometry and local differential geometry},
   JOURNAL = {Ann. Sci. \'Ecole Norm. Sup. (4)},
  FJOURNAL = {Annales Scientifiques de l'\'Ecole Normale Sup\'erieure.
              Quatri\`eme S\'erie},
    VOLUME = {12},
      YEAR = {1979},
    NUMBER = {3},
     PAGES = {355--452},
      ISSN = {0012-9593},
     CODEN = {ENAQAF},
   MRCLASS = {53A20 (14C21 53A60)},
  MRNUMBER = {559347 (81k:53004)},
MRREVIEWER = {M. A. Akivis},
       URL = {http://www.numdam.org/item?id=ASENS_1979_4_12_3_355_0},
}
	
\bib{HNP09}{article}{
label={HNP09},
    AUTHOR = {C. Haase},  
    author={B. Nill},
    author={S. Payne},
     TITLE = {Cayley decompositions of lattice polytopes and upper bounds
              for {$h^*$}-polynomials},
   JOURNAL = {J. Reine Angew. Math.},
  FJOURNAL = {Journal f\"ur die Reine und Angewandte Mathematik. [Crelle's
              Journal]},
    VOLUME = {637},
      YEAR = {2009},
     PAGES = {207--216},
}
\bib{HA}{book}{
label={HA},
    AUTHOR = {R. Hartshorne},
     TITLE = {Algebraic geometry},
      NOTE = {Graduate Texts in Mathematics, No. 52},
 PUBLISHER = {Springer-Verlag},
   ADDRESS = {New York},
      YEAR = {1977},
     PAGES = {xvi+496},
      ISBN = {0-387-90244-9},
   MRCLASS = {14-01},
  MRNUMBER = {0463157 (57 \#3116)},
MRREVIEWER = {Robert Speiser},
}

\bib{LM00}{article}{
label={LM00},
    AUTHOR = {A. Lanteri, Antonio}
    author={R. Mallavibarrena},
     TITLE = {Higher order dual varieties of generically {$k$}-regular
              surfaces},
   JOURNAL = {Arch. Math. (Basel)},
  FJOURNAL = {Archiv der Mathematik},
    VOLUME = {75},
      YEAR = {2000},
    NUMBER = {1},
     PAGES = {75--80},
      ISSN = {0003-889X},
     CODEN = {ACVMAL},
   MRCLASS = {14C20 (14J25 14J60)},
  MRNUMBER = {1764895 (2001f:14015)},
MRREVIEWER = {Andrew J. Sommese},
       DOI = {10.1007/s000130050476},
       URL = {http://dx.doi.org/10.1007/s000130050476},
}

\bib{L94}{article}{
label={L94},
    AUTHOR = {J.M. Landsberg},
     TITLE = {On second fundamental forms of projective varieties},
   JOURNAL = {Invent. Math.},
  FJOURNAL = {Inventiones Mathematicae},
    VOLUME = {117},
      YEAR = {1994},
    NUMBER = {2},
     PAGES = {303--315},
      ISSN = {0020-9910},
     CODEN = {INVMBH},
   MRCLASS = {14N05},
  MRNUMBER = {1273267 (95g:14057)},
MRREVIEWER = {P. Abellanas},
       DOI = {10.1007/BF01232243},
       URL = {http://dx.doi.org/10.1007/BF01232243},
}

\bib{MT11}{article}{
label={MT11},
    AUTHOR = {Y. Matsui},
    author={K. Takeuchi},
     TITLE = {A geometric degree formula for {$A$}-discriminants and {E}uler
              obstructions of toric varieties},
   JOURNAL = {Adv. Math.},
  FJOURNAL = {Advances in Mathematics},
    VOLUME = {226},
      YEAR = {2011},
    NUMBER = {2},
     PAGES = {2040--2064},
      ISSN = {0001-8708},
     CODEN = {ADMTA4},
   MRCLASS = {14M25},
  MRNUMBER = {2737807 (2012e:14103)},
       DOI = {10.1016/j.aim.2010.08.020},
       URL = {http://dx.doi.org/10.1016/j.aim.2010.08.020},
}
\bib{MU02}{article}{
label={MU02},
    AUTHOR = {M. Mustata},

     TITLE = {Vanishing theorems on toric varieties},
   JOURNAL = {Tohoku Math. J.},
    VOLUME = {54},
      YEAR = {2002},
    NUMBER = {3},
     PAGES = {4451--470},
}


\bib{ODA}{incollection}{
label={ODA},
    AUTHOR = {T. Oda},
     TITLE = {Convex bodies and algebraic geometry---toric varieties and
              applications. {I}},
 BOOKTITLE = {Algebraic {G}eometry {S}eminar ({S}ingapore, 1987)},
     PAGES = {89--94},
 PUBLISHER = {World Sci. Publishing},
   ADDRESS = {Singapore},
      YEAR = {1988},
   MRCLASS = {14L32 (52A25 52A43)},
  MRNUMBER = {966447 (89m:14026)},
MRREVIEWER = {G. Ewald},
}

\bib{ODAb}{book} {
label={ODAb}
    AUTHOR = {T. Oda},
     TITLE = {Convex bodies and algebraic geometry},
    SERIES = {Ergebnisse der Mathematik und ihrer Grenzgebiete (3) [Results
              in Mathematics and Related Areas (3)]},
    VOLUME = {15},
      NOTE = {An introduction to the theory of toric varieties,
              Translated from the Japanese},
 PUBLISHER = {Springer-Verlag},
   ADDRESS = {Berlin},
      YEAR = {1988},
     PAGES = {viii+212},
      ISBN = {3-540-17600-4},
   MRCLASS = {14L32 (14-02 52A25 52A43)},
  MRNUMBER = {922894 (88m:14038)},
MRREVIEWER = {I. Dolgachev},
}

\bib{Pa13}{article}{
label={Pa13},
author={A. Paffenholz},
title={Finiteness of the polyhedral Q-codegree spectrum},
journal={	arXiv:1301.4967 [math.CO]},}


\bib{Re83}{incollection}{
label={Re83},
    AUTHOR = {M. Reid},
     TITLE = {Decomposition of toric morphisms},
 BOOKTITLE = {Arithmetic and Geometry},
     PAGES = {395--418},
 PUBLISHER = {Progress in Math. 36, Birkh\"auser},
   ADDRESS = {Boston},
      YEAR = {1983},
   }



\end{biblist}
\end{bibdiv}
\end{document}